\documentclass{amsart}

\usepackage[sort&compress, round]{natbib}
\usepackage{amsmath, amssymb, amsthm, amscd, verbatim}
\usepackage{enumitem}
\usepackage[usenames,dvipsnames]{color}
\usepackage{tikz-cd}

\newcommand{\R}  {{\mathbb R}}
\newcommand{\C}  {{\mathbb C}}
\newcommand{\K}  {{\mathbb K}}
\newcommand{\N}  {{\mathbb N}}

\newcommand{\bs} {{\boldsymbol{\sigma}}}
\newcommand{\bo} {{\boldsymbol{\omega}}}
\newcommand{\ba} {{\boldsymbol{\alpha}}}
\newcommand{\bn} {{\boldsymbol{\nu}}}
\newcommand{\bx} {{\boldsymbol{x}}}
\newcommand{\by} {{\boldsymbol{y}}}
\newcommand{\bu} {{\boldsymbol{u}}}
\newcommand{\bv} {{\boldsymbol{v}}}
\newcommand{\bw} {{\boldsymbol{w}}}
\newcommand{\be} {{\boldsymbol{1}}}
\newcommand{\bof}{{\boldsymbol{f}}}
\newcommand{\bog}{{\boldsymbol{g}}}
\newcommand{\bN} {{\boldsymbol{N}}}
\newcommand{\bU} {{\boldsymbol{U}}}
\newcommand{\bUJ}{{\boldsymbol{U}_{\!\!J}}}

\newcommand{\fe} {{\mathfrak{e}}} 
\newcommand{\ff} {{\mathfrak{f}}} 
\newcommand{\fg} {{\mathfrak{g}}} 
\newcommand{\fu} {{\mathfrak{u}}}
\newcommand{\fv} {{\mathfrak{v}}}
\newcommand{\X}  {{\mathfrak X}}
\newcommand{\Y}  {{\mathfrak Y}}
\newcommand{\Z}  {{\mathfrak Z}}

\newcommand{\HH} {{\mathcal H}}
\newcommand{\KK} {{\mathcal{K}}}

\newcommand{\Co} {C}
\newcommand{\Ko} {\KK}
\newcommand{\Coo}{C_{\star}}
\newcommand{\Koo}{\KK_{\star}}
\newcommand{\Xo} {\X}

\newcommand{\eps}  {\varepsilon}
\renewcommand{\phi}{\varphi}

\newcommand{\scp}[2]{\langle #1, #2 \rangle}
\newcommand{\spann} {\operatorname{span}}

\newcommand{\ha}{H1}
\newcommand{\hb}{H2}
\newcommand{\hc}{H3}

\numberwithin{equation}{section}

\theoremstyle{plain}
\newtheorem{lemma}{Lemma}[section]
\newtheorem{theo}[lemma]{Theorem}
\newtheorem{prop}[lemma]{Proposition}
\theoremstyle{definition}
\newtheorem{rem}[lemma]{Remark}
\newtheorem{exmp}[lemma]{Example}
\newtheorem{defi}[lemma]{Definition}

\begin{document}

\title[]
{Countable Tensor Products of Hermite Spaces and Spaces 
of Gaussian Kernels}

\author[Gnewuch]
{M.~Gnewuch}
\address{
Institut f\"ur Mathematik\\
Universit\"at Osnabr\"uck\\
Albrechtstra{\ss}e \ 28A\\
49076 Osnabr\"uck\\ 
Germany}
\email{michael.gnewuch@uni-osnabrueck.de}

\author[Hefter]
{M.~Hefter }
\address{Fachbereich Mathematik\\
Technische Universit\"at Kaisers\-lautern\\
Postfach 3049\\
67653 Kaiserslautern\\
Germany}
\email{hefter@mathematik.uni-kl.de}

\author[Hinrichs]
{A.~Hinrichs}
\address{
Institut f\"ur Analysis\\
Johannes-Kepler-Universit\"at Linz\\
Altenberger Str.\ 69\\
4040 Linz\\
Austria}
\email{aicke.hinrichs@jku.at}

\author[Ritter]
{K.~Ritter}
\address{Fachbereich Mathematik\\
Technische Universit\"at Kaisers\-lautern\\
Postfach 3049\\
67653 Kaiserslautern\\
Germany}
\email{ritter@mathematik.uni-kl.de}

\date{February 22, 2022}

\begin{abstract}
In recent years finite tensor products of
reproducing kernel Hilbert spaces (RKHSs) 
of Gaussian kernels on the one hand 
and of Hermite spaces on the other hand
have been considered in tractability analysis of
multivariate problems. 
In the present paper we study countably infinite tensor products
for both types of spaces. We show that the incomplete tensor
product in the sense of von Neumann may be identified with an RKHS whose 
domain is a proper subset of the sequence space $\R^\N$.
Moreover, we show that each tensor product of spaces  of Gaussian 
kernels having square-summable shape parameters is isometrically 
isomorphic to a tensor product of Hermite spaces; the corresponding 
isomorphism is given explicitly, respects point evaluations,
and is also an $L^2$-isometry. 
This result directly transfers to the case of finite tensor products.
Furthermore, we provide regularity results 
for Hermite spaces of functions of a single variable.
\end{abstract}

\maketitle

\section{Introduction}

This paper is motivated by the study of integration and
$L^2$-approximation in a tensor product setting for problems
with a large or infinite number $d$ of variables. Specifically,
we consider Hilbert spaces $H(M)$ with reproducing kernels
\[
M := \bigotimes_j m_j
\]
of tensor product form. Here we have
$j \in \{1,\dots,d\}$ if $d \in \N$ and $j \in \N$ if $d=\infty$,
and we assume that either all of the kernels 
\[
m_j \colon \R \times \R \to \R
\]
are Hermite or all are Gaussian kernels.  The same
terminology, Hermite kernel or Gaussian kernel,
will be used for the corresponding tensor 
product kernels $M$, and the reproducing kernel Hilbert spaces (RKHSs) 
with Hermite kernels are known as Hermite spaces, see \citet{IL15}.
Analogously, and despite a different use in stochastic analysis,
the RKHSs of Gaussian kernels will be called Gaussian spaces
throughout this paper. 

For both types of kernels, Hermite and Gaussian, the functions
in $H(M)$ are defined on the space $\R^d$ if 
$d \in \N$ or on a subset of the sequence space $\R^\N$ if $d=\infty$. 
Moreover, the measure $\mu$ that defines the integral or 
the $L^2$-norm is the corresponding 
product of the standard normal distribution $\mu_0$.

We are primarily interested in the case $d=\infty$.
This is the technically more
demanding case from which corresponding results
for the case $d\in\N$ can be deduced easily.
In the present paper we provide an analysis of the function spaces,
while the complexity of integration and $L^2$-approximation will
be studied in a follow-up paper.

Our approach is as follows. At first, we study the spaces $H(m_j)$
of functions of a single variable, and then we turn to the countable 
tensor product of these spaces, which allows to carry over the results
for $d=1$ to the tensor product Hilbert space 
in a canonical way.
Finally, we identify the tensor product Hilbert space with an
RKHS $H(M)$ for the tensor product kernel
\[
M \colon \X \times \X \to \R
\]
on a suitably chosen domain 
\[
\X \subseteq \R^\N.
\]

Let us refer to prior work on tractability of integration and 
$L^2$-approximation in the case $d \in \N$, i.e., for finite tensor 
products of either Hermite or Gaussian kernels.
Both types of kernels have been studied separately, so far.
For Hermite spaces the first contribution is due to
\citet{IL15}, followed by
\citet{IKLP15}, 
\citet{IKPW16a},
and
\citet{IKPW16b} 
with a focus on spaces of analytic functions,
and by 
\citet{DILP18}
with a focus on spaces of finite smoothness.
For Gaussian spaces the first contribution is due to
\citet{FHW2012}, followed by 
\citet{KSW2017},
\citet{SW2018}, and
\citet{KOG21}.
To the best of our knowledge, the case $d=\infty$ has not been
studied yet neither for Hermite nor for Gaussian kernels.

We show, in particular, that Hermite spaces and Gaussian spaces
with $d=\infty$ are closely related in the following sense:
Every Gaussian space with square-summable shape parameters
is isometrically isomorphic to a Hermite space with an explicitly
given isometry that respects point evaluations and is also an
$L^2(\mu)$-isometry. We add that the same holds true for every $d
\in \N$. This result may explain an observation due to 
\citet[p.~2]{KOG21} 
concerning the analysis of integration and approximation problems,
namely that techniques
similar to those used in the Gaussian case are used in the Hermite
case.
A similar observation was made earlier by \citet[p.~830]{KSW2017},
were the authors write that Gauss-Hermite quadratures used in
\citet{IKLP15} for integrands stemming from Hermite spaces work well 
also on Gaussian spaces, although the latter ones do not contain any 
non-trivial polynomial.

In the following we present the construction and a sketch of results for 
Hermite spaces. To define the Hermite kernels $m_j := k_j$
we employ the $L^2(\mu_0)$-normalized Hermite 
polynomials $h_\nu$ of degree $\nu \in \N_0$,
see Section~\ref{SUBSEC:Hermite_1} for the definition and 
basic properties.
For any choice of Fourier
weights $\alpha_{\nu,j} > 0$ with 
\[
\inf_{\nu \in \N_0} \alpha_{\nu,j} > 0
\]
and
\[
\sum_{\nu \in \N_0} \alpha_{\nu,j}^{-1} \cdot |h_\nu(x)|^2 < \infty
\]
for all $x \in \R$ the corresponding Hermite kernel is given by
\[
\phantom{\qquad\quad x,y \in \R.}
k_j (x,y) := 
\sum_{\nu\in \N_0} 
\alpha_{\nu,j}^{-1} \cdot h_\nu (x) \cdot h_\nu (y),
\qquad\quad x,y \in \R.
\]
A general set (\ha)--(\hc) of summability and monotonicity
assumptions for the Fourier weights is presented in
Section~\ref{s3.2}. In this introduction we focus on two particular 
kinds of Fourier 
weights, namely the case of polynomial growth, where 
\[
\phantom{\qquad\quad \nu,j \in \N,}
\alpha_{\nu,j} := (\nu+1)^{r_j},
\qquad\quad \nu,j \in \N,
\]
with $r_j > 0$, and the case of (sub-)exponential
growth, where
\[
\phantom{\qquad\quad \nu,j \in \N,}
\alpha_{\nu,j} := 2^{r_j \cdot \nu^{b_j}},
\qquad\quad \nu,j \in \N,
\]
with $r_j, b_j > 0$. The latter kind of Fourier weights with 
$b_j=1$ for all $j\in\N$ is the proper choice when we study the
relation between Hermite and Gaussian spaces,
see Section~\ref{s:iso}.

Let us consider Hermite spaces of functions of a
single variable. For any fixed $j \in \N$ the regularity of the functions 
from $H(k_j)$ is determined by the asymptotic behavior of 
$\alpha_{\nu,j}$ as $\nu \to \infty$, see 
\citet{IL15}, \citet{IKLP15}, and  \citet{DILP18}
for specific results. For Fourier weights 
with a polynomial growth the low regularity limit is $r_j = 1/2$, where
we loose convergence of $k_j$. See Example~\ref{r20} and cf.\
\citet[Sec.~2]{DILP18}, where convergence is established for $r_j > 5/6$.  
For Fourier weights 
with a (sub-)exponential growth the elements of $H(k_j)$ belong to the 
Gevrey class of index $\max(1/(2b_j),1)$; in particular, they are real 
analytic functions if $b_j \geq 1/2$. See Lemma~\ref{l76} and
cf.\ \citet[Prop.~3.7]{IL15} and \citet[Prop.~3]{IKLP15}, where
analyticity is established for $b_j \geq 1$, as well as 
\citet[Rem.~1]{IKPW16b}, where for $b_j\in {]0,1[}$ the inclusion 
of $H(k_j)$ in the Gevrey class of index $1/b_j$ is claimed.

Furthermore, Hermite spaces with Fourier weights of polynomial 
growth with $r_j$ varying in ${]1/2,\infty[}$ and of exponential 
growths with $b_j=1$ and with $r_j$ varying in ${]0,\infty[}$ 
form interpolation scales with equality of the norms due to quadratic 
interpolation and the original norm on the RKHS.
See Remark~\ref{r30} and Examples~\ref{r20} and \ref{r22}. 

Next we present the main results for Hermite spaces of functions of 
infinitely many variables, where we require 
\[
r_j > 1/2
\qquad \text{and} \qquad
\sum_{j \in \N} 2^{-r_j} < \infty
\]
in the case of polynomial growth
and
\[
\sum_{j \in \N} 2^{-r_j} < \infty
\qquad \text{and} \qquad
\inf_{j \in \N}b_j > 0
\]
in the case of (sub-)exponential growth.
In both cases we additionally require
\[
\sum_{j \in \N} |\alpha_{0,j}-1| < \infty.
\]
See Section~\ref{s3.3}.
We show that the maximal domain $\X$ for the tensor product kernel
$K := \bigotimes_{j \in \N} k_j$ is related to an intersection of 
weighted $\ell^2$-spaces and satisfies
\[
\ell^\infty \subsetneq \X \subsetneq \R^\N
\]
as well as
\[
\mu(\X) = 1.
\]
See Proposition~\ref{lem:domain2} and Lemma~\ref{l2}. 
We conclude, in particular, that the Hermite spaces $H(K)$ of
functions of infinitely many variables are never spaces of
functions on the domain $\R^\N$. 

In the case of exponential growth with 
\[
\inf_{j \in \N} b_j \geq 1
\]
the maximal domain is explicitly given as
\[
\X = \bigg\{ \bx \in \R^\N \colon \sum_{j \in \N} 2^{-r_j}
\cdot x_j^2 < \infty  \bigg\}. 
\]
See Proposition~\ref{Lemma:Char_X_EG} and cf.\ 
Proposition~\ref{p17}, which deals with the case of polynomial and 
sub-exponential growth.

Moreover, we show that the incomplete tensor product $H^{(\ba_0^{-1/2})}$, 
see \citet{Neu39} 
and Definition \ref{def:incprod},
of the Hermite spaces $H(k_j)$ based on the constant functions  
$\alpha_{0,j}^{-1/2}$ as unit vectors may be identified with the Hermite space 
$H(K)$ on the maximal domain $\X$ in a canonical way. See Lemma~\ref{l51} 
and Theorem~\ref{ta1}. 
We add that the tensor product norm on $H(K)$ respects the 
infinite-dimensional ANOVA decomposition of a function, i.e., 
different ANOVA components are pairwise orthogonal in $H(K)$. 
See Remark~\ref{rem:ANOVA}.

Next we come to Gaussian spaces. The Gaussian kernel 
$m_j := \ell_j$ with shape parameter $\sigma_j > 0$ is given by
\[
\phantom{\qquad\quad x,y \in \R.}
\ell_j (x,y) := \exp( -\sigma_j^2 \cdot (x-y)^2),
\qquad\quad x,y \in \R.
\]
Gaussian spaces $H(\ell_j)$ and finite tensor products thereof
are thoroughly analyzed in 
\citet{SHS2006}, see also \citet[Sec.~4.4]{SC08}. 
In the following we sketch our main results for Gaussian
spaces in the case $d=\infty$ with shape parameters
$\sigma_j>0$ satisfying
\[
\sum_{j \in \N} \sigma_j^2 < \infty.
\]
We show that the incomplete tensor product
$G^{(\bv)}$ of the Gaussian spaces $H(\ell_j)$ based on the unit
vectors $v_j(x) := \exp(-\sigma_j^2 \cdot x^2)$ may be identified
with a Gaussian space $H(L)$ in a canonical way.
The corresponding domain for the tensor product kernel $L :=
\bigotimes_{j \in \N} \ell_j$ is given by
\[
\X = \bigg\{ \bx \in \R^\N \colon \sum_{j \in \N} \sigma_j^2 \cdot x_j^2
< \infty \bigg\},
\]
which is a proper subset of $\R^\N$,
despite the convergence of $\prod_{j=1}^\infty \ell_j(x_j,y_j)$
for all $\bx, \by \in \R^\N$. 
See Lemma \ref{l11}, Remark~\ref{rem:summarize}, and 
Theorem \ref{ta1}. We add that $\mu(\X) = 1$.

The identification of the Gaussian spaces $H(L)$ and the Hermite
spaces $H(K)$ with incomplete tensor products allows to establish
the following result via tensorization.
For any Gaussian space $H(L)$ we consider the Hermite space $H(K)$ with 
Fourier weights
\[
\phantom{\qquad\quad \nu \in \N_0,\ j \in \N,}
\alpha_{\nu, j} := \frac{1}{(1-\beta_j)\cdot \beta^\nu_j},
\qquad\quad \nu \in \N_0,\ j \in \N,
\]
where 
\[
\beta_j := 1 - \frac{2}{1+c_j^2}
\qquad \text{with} \qquad
c_j := \left(1+8\sigma_j^2\right)^{1/4}.
\]
Moreover, we define 
\[
Qf (\bx) := \prod_{j=1}^\infty c_j
\cdot 
\exp \biggl( - \sum_{j\in\N} \frac{c_j^2-1}{4} \,x_j^2 \biggr) 
\cdot f\bigl(c_1 x_1, c_2 x_2, \dots \bigr)
\]
for $f \in L^2(\mu)$ and $\bx \in \X$.
Then $Q$ is an isometric isomorphism on $L^2(\mu)$ and 
$Q|_{H(K)}$ is an isometric isomorphism between $H(K)$ and $H(L)$.
See Theorem~\ref{t1}.
Another feature of $Q$ that is highly relevant
in the context of integration and approximation problems
is the obvious fact that for every $\bx \in \X$ a single function value 
of $f \in H(K)$ suffices to compute $Qf(\bx)$ and vice versa.

This paper is organized as follows: 
Since all spaces of functions we consider are subspaces of (suitable) 
spaces of square-integrable functions, we discuss in 
Section~\ref{SEC:L2_Sub} under which conditions subspaces of $L^2$-spaces 
are reproducing kernel Hilbert spaces and what
their kernels and their 
norms look like. Proposition~\ref{l24} provides sufficient
conditions and facilitates our later analysis of function spaces 
considerably. In Section~\ref{SEC:Hermite} we define the Hermite spaces 
we want to study and present their essential properties. 
There and in 
the following two sections we first consider spaces of univariate 
functions and then turn to spaces of functions depending on infinitely 
many variables. 
The important example cases, where the 
Fourier weights exhibit polynomial or (sub-)exponential growth, are 
studied in Section~\ref{s3.3}.
In Section~\ref{SEC:Gauss} we present the essential well-known facts on 
Gaussian spaces
of univariate functions, see 
Section~\ref{SUBSEC:Gauss_1},
and then analyze Gaussian spaces of functions of infinitely many 
variables, see 
Section~\ref{SUBSEC:Gauss_inf}.
Relations between Hermite and Gaussian spaces are 
established in Section~\ref{s:iso}.
In the appendix we recall the 
construction and key properties of complete and incomplete
tensor products of Hilbert spaces and discuss the two important cases of 
incomplete tensor products of reproducing kernel Hilbert spaces, see 
Section~\ref{as4}, and of $L^2$-spaces, see Section~\ref{as5}.

Let us close the introduction with some remarks concerning 
notation, terminology,  and conventions.
We consider the Borel $\sigma$-algebra on $\R$ and, in particular,
the standard normal distribution $\mu_0$ on this space.
Moreover, we consider the corresponding product $\sigma$-algebra on 
$\R^\N$ and, in particular, the countable product $\mu$
of $\mu_0$ on this space. 
For $E \neq \emptyset$ and $\K \in \{\R,\C\}$ a function
$K \colon E \times E \to \K$ is a called a 
\emph{reproducing kernel}
if $K$ is symmetric in the case $\K=\R$ and Hermitian in the case
$\K=\C$ and positive definite in the following sense:
For all $n \in \N$, $x_1,\dots,x_n \in
E$, $a_1,\dots,a_n \in \K$ we have $\sum_{i,j=1}^n \overline{a_i}
a_j \cdot K(x_i,x_j) \geq 0$.

For any sequence $(a_j)_{j \geq J_0}$ in $\R$ such that
$(\prod_{j=J_0}^J a_j)_{J\geq J_0}$ is convergent we put 
\begin{equation}\label{g29}
\prod_{j=J_0}^\infty a_j := \lim_{J \to \infty} \prod_{j=J_0}^J a_j.
\end{equation}
The collection of arbitrary finite partial products
of $(a_j)_{j \geq J_0}$ forms a net; the corresponding stronger
notion of convergence is discussed in Section~\ref{as1}.

We use the following notation for subspaces of $\R^\N$.
The space of all sequences $\bx := (x_j)_{j \in \N}$ such that 
$\{j \in \N \colon x_j \neq 0\}$ is finite is denoted by
$c_{00}$, and $\ell^\infty$ denotes the space of all bounded sequences.
For any sequence $\bo := (\omega_j)_{j \in \N}$ of positive real numbers
and $1 \leq p < \infty$ we use $\ell^p(\bo)$ to denote
the space of all sequences $\bx$ such that
\[
\sum_{j \in \N} \omega_j \cdot |x_j|^p < \infty.
\]
For $n \in \N$ we use $\bN_n$ to denote the set of all 
sequences $(\nu_j)_{j \in \N}$ in $\N_0$ such that 
$\nu_j = 0$ for every $j > n$. Moreover, we put 
$\bN := \bigcup_{n \in \N} \bN_n$, i.e., $\bN = c_{00} \cap \N_0^\N$.

We use $\bU$ to denote the set of all finite subsets of $\N$.

\section{$L^2$-Subspaces as Reproducing Kernel Hilbert Spaces}
\label{SEC:L2_Sub}

Consider a space $L^2(\rho) := L^2(E,\rho)$ over $\K \in \{\R,\C\}$ with 
respect to a measure $\rho$ on a $\sigma$-algebra on a set $E$. 
Let $N$ denote a countable set, and let $(\fe_\nu)_{\nu \in N}$ denote a 
family of square-integrable functions on $E$ such that the corresponding 
equivalence classes $e_\nu$ form an orthonormal basis of $L^2(\rho)$. 
Moreover, let $(\alpha_\nu)_{\nu \in N}$ denote a family of 
positive real numbers, which will be called \emph{Fourier weights} in 
view of the following result
that is essentially due to \citet[Sec.~2.1]{GHHRW2020}.

\begin{prop}\label{l24}
Assume that
\begin{equation}\label{g42}
\inf_{\nu \in N} \alpha_\nu > 0
\end{equation}
and
\begin{equation}\label{g23}
\sum_{\nu\in N} \alpha_\nu^{-1} \cdot |\fe_\nu (x)|^2 < \infty
\end{equation}
for every $x \in E$. If $(c_\nu)_{\nu \in N} \in \K^N$ satisfies 
\[
\sum_{\nu \in N} \alpha_{\nu} \cdot |c_\nu|^2 < \infty,
\]
then $\sum_{\nu \in N} c_\nu \cdot e_\nu$ converges in
$L^2(\rho)$ and
$\sum_{\nu \in N} c_\nu \cdot \fe_\nu(x)$ is absolutely convergent
for every $x \in E$. Moreover,
\[
\phantom{\qquad\quad x,y \in E,}
K (x,y) := \sum_{\nu\in N} 
\alpha_\nu^{-1} \cdot \fe_\nu (x) \cdot \overline{\fe_\nu (y)},
\qquad\quad x,y \in E,
\]
defines a reproducing kernel on 
$E \times E$. The corresponding Hilbert space is given by
\begin{equation}\label{eq:charRKHS}
H(K) = \Bigl\{
\sum_{\nu \in N} c_\nu \cdot \fe_\nu \colon 
\text{$(c_\nu)_{\nu \in N} \in \K^N$ with 
$\sum_{\nu \in N} \alpha_{\nu} \cdot |c_\nu|^2 < \infty$} 
\Bigr\}
\end{equation}
and 
\[
\phantom{\qquad \ff,\fg \in H(K).}
\scp{\ff}{\fg}_{H(K)} = 
\sum_{\nu \in N} \alpha_{\nu} \cdot 
\int_E \ff \cdot \overline{\fe_\nu} \, d \rho \cdot
\int_E \fe_\nu \cdot \overline{\fg} \, d \rho, 
\qquad \ff,\fg \in H(K).
\]
\end{prop}

\begin{proof}
The convergence of 
$\sum_{\nu \in N} c_\nu \cdot e_\nu$ and
$\sum_{\nu \in N} c_\nu \cdot \fe_\nu(y)$, as claimed, is easily
verified.

The remaining part of the proof is based on the following
facts. The linear subspace 
\begin{equation}\label{g25a}
\HH := 
\Bigl\{ f \in L^2(\rho) \colon 
\sum_{\nu \in N} \alpha_{\nu} \cdot |\scp{f}{e_\nu}_{L^2(\rho)}|^2 
< \infty \Bigr\}
\end{equation}
of $L^2(\rho)$,
equipped with the scalar product
\begin{equation}\label{g25b}
\phantom{\qquad f,g \in \HH,}
\scp{f}{g}_{\HH} := 
\sum_{\nu \in N} \alpha_{\nu} \cdot 
\scp{f}{e_\nu}_{L^2(\rho)} \cdot \scp{e_\nu}{g}_{L^2(\rho)},
\qquad f,g \in \HH,
\end{equation}
is a Hilbert space with a continuous embedding of norm
$\sup_{\nu \in N} \alpha_\nu^{-1/2}$ into $L^2(\rho)$. 
The elements $h_\nu := \alpha_\nu^{-1/2} e_\nu$ with $\nu \in N$ form an
orthonormal basis of $\HH$.

We employ the first part of the proof to define a linear
mapping $\Phi \colon \HH \to \K^E$ by
\[
\phantom{\qquad\quad x \in E,}
\Phi f (x) := 
\sum_{\nu \in N} \scp{f}{e_\nu}_{L^2(\rho)} \cdot \fe_\nu (x),
\qquad\quad x \in E,
\]
i.e., $\Phi f$ is the pointwise
limit of the $L^2$-Fourier partial sums of $f \in \HH$.
Obviously, $\Phi h_\nu = \alpha_\nu^{-1/2} \fe_\nu$.
The mapping $\Phi$ is injective, see
\citet[Rem.~2.6]{GHHRW2020}, so that 
$\Phi(\HH)$, equipped with the scalar product
\[
\phantom{\qquad\quad f,g \in \HH,}
\scp{\Phi f}{\Phi g} = \scp{f}{g}_\HH,
\qquad\quad f,g \in \HH,
\]
is a Hilbert space of $\K$-valued functions on $E$.
It follows that $(\Phi(\HH), \scp{\cdot}{\cdot})$ is a reproducing
kernel Hilbert space, see
\citet[Lem.~2.2]{GHHRW2020}, 
and its reproducing kernel is $K$, 
see \citet[Lem.~2.1]{GHHRW2020}. 
Finally, observe that
\[
\Phi(\HH) = 
\Bigl\{
\sum_{\nu \in N} c_\nu \cdot \fe_\nu \colon 
\text{$(c_\nu)_{\nu \in N} \in \K^N$ with 
$\sum_{\nu \in N} \alpha_{\nu} \cdot |c_\nu|^2 < \infty$}
\Bigr\}.
\qedhere
\]
\end{proof} 

Subsequently, we do no longer distinguish between 
square-integrable functions on $E$ and elements of $L^2(\rho)$;
in particular, we identify $\fe_\nu$ and $e_\nu$.

\begin{rem}\label{r42}
Assume that \eqref{g42} and \eqref{g23}
for every $x \in E$ are satisfied. 
According to Proposition~\ref{l24},
$H(K) \subseteq L^2(\rho)$ with a continuous embedding,
which is compact if and only if for every $r>0$ the set 
$\{ \nu \in N \colon \alpha_\nu \le r \}$ is finite.
The proof of the proposition reveals that
$\HH$ is a feature space and $\phi \colon E \to \HH$,
given by $\phi(x) := \sum_{\nu \in \N} \alpha_\nu^{-1} \cdot
e_\nu(x) \cdot e_\nu$, is a feature map of $K$.
See, e.g., \citet[Sec.~4.1]{SC08} for these concepts.

Additionally, assume that 
\[
e_{\nu^*} = 1
\]
for some $\nu^* \in N$.
Then 
\[
K^* := K - \alpha_{\nu^*}^{-1}
\]
is a reproducing kernel, too. Furthermore, 
\[
\int_E f \, d \rho = 0
\]
for every $f \in H(K^*)$, since 
$\int_E e_\nu \, d \rho = \scp{e_\nu}{e_{\nu^*}}_{L^2(\rho)} = 0$
for $\nu \neq \nu^*$. In particular, this yields 
\[
H(1) \cap H(K^*) = \{0\},
\]
and $\int_E f \, d \rho$ is the orthogonal projection 
of $f \in H(K)$ onto $H(1)$ in $H(K)$ as well as in $L^2(\rho)$.
\end{rem}

\begin{rem}\label{r30}
We discuss interpolation of Hilbert spaces in the present
framework. To this end we consider two families 
$(\alpha_{\nu,0})_{\nu \in N}$ and
$(\alpha_{\nu,1})_{\nu \in N}$ of Fourier weights such that 
$\inf_{\nu \in N} \alpha_{\nu,0} > 0$ and
$\{ \nu \in N \colon \alpha_{\nu,1}/\alpha_{\nu,0} \leq r\}$ is
finite for every $r > 0$. Clearly
$\inf_{\nu \in N} \alpha_{\nu,1} > 0$, too.
Analogously to \eqref{g25a} and \eqref{g25b}
we obtain
linear subspaces $\HH_0$ and $\HH_1$ of $L^2(\rho)$, based on
$(\alpha_{\nu,0})_{\nu \in N}$ and
$(\alpha_{\nu,1})_{\nu \in N}$,
respectively. It follows that
$\HH_1$ is a dense subspace of $\HH_0$ with a compact
embedding. 

For quadratic interpolation of the spaces $\HH_0$ and $\HH_1$ by means of 
the $K$-method and the $J$-method we obtain
\[
K_{\theta,2}(\HH_0,\HH_1) = J_{\theta,2}(\HH_0,\HH_1) = \HH
\]
for every $0 < \theta < 1$, where $\HH$ is the linear subspace of
$L^2(\rho)$ based on the Fourier weights
\[
\phantom{\qquad\quad \nu \in \N.}
\alpha_\nu := \alpha_{\nu,1}^\theta \cdot
\alpha_{\nu,0}^{1-\theta},
\qquad\quad \nu \in \N.
\]
Furthermore, we have equality of norms in these spaces.
See \citet[Thm.~3.4]{ChaEtAl2015}.
Assume, additionally, that \eqref{g23} is satisfied for
$\alpha_\nu:=\alpha_{\nu,0}$. Then the same conclusions hold true
for the associated reproducing kernel Hilbert spaces.
\end{rem}

\section{Hermite Spaces}
\label{SEC:Hermite}

\subsection{Functions of a Single Variable}
\label{SUBSEC:Hermite_1}

Here we consider the space
$L^2(\mu_0) := L^2(\R, \mu_0)$ for the standard normal
distribution $\mu_0$ over $\K=\R$. 

The normalized Hermite polynomials $h_\nu$ of degree 
$\nu\in \N_0$ form an orthonormal basis of $L^2(\mu_0)$ 
and can be obtained by orthonormalizing the monomials. In particular, 
we have $h_0(x)=1$ and $h_1(x) = x$ for all $x\in\R$.
An explicit representation for any $\nu \in \N_0$ is 
\begin{equation}\label{eq:rep_hermite_polynomial}
\phantom{\qquad\quad x\in \R.}
h_\nu(x) = \sqrt{\nu !} \sum^{\lfloor \nu/2 \rfloor}_{k=0} 
(-1)^k \frac{x^{\nu - 2k}}{2^k\, k! \, (\nu -2k)! },
\qquad\quad x\in \R,
\end{equation}
see \citet[Eqn.~(5.5.4)]{S75}.
For basic properties of
these functions we refer to, e.g., \citet[Sec.~1.3]{MR1642391}. 
In the next lemma we collect some asymptotic properties of the 
Hermite polynomials. 

\begin{lemma}\label{l41}
For every $x\in\R$ we have
\begin{align}\label{eq10}
\sup_{\nu\in\N_0} \bigl(\nu^{1/4} \cdot |h_\nu(x)|\bigr) < \infty
\end{align}
as well as Cram\'er's inequality 
\begin{align}\label{eq11}
\sup_{\nu \in \N_0} |h_\nu(x)| \leq \exp(x^2/4). 
\end{align}
For $x=0$ we have
\begin{align}\label{eq12}
\inf_{\nu\in\N} \bigl(\nu^{1/4} \cdot |h_{2\nu}(x)|\bigr) > 0. 
\end{align}
\end{lemma}

\begin{proof}
For every $x\in\R$ we have
\begin{equation}\label{eq1}
\sup_{\nu\in\N_0}
\frac{\Gamma(\nu/2+1)}{\Gamma(\nu+1)}
\cdot 2^{\nu/2}\cdot \sqrt{\nu!}\cdot |h_\nu(x)|
<\infty,
\end{equation}
see, e.g., \citet[Eqn.~(8.22.8)]{S75}.
Moreover, Stirling's approximation yields 
the strong asymptotics
\begin{equation}\label{eq2}
\frac{\Gamma(\nu/2+1)}{\Gamma(\nu+1)}\cdot 2^{\nu/2}\cdot \sqrt{\nu!}
\approx
(\pi/2)^{1/4}\cdot \nu^{1/4}
\end{equation}
as $\nu\to\infty$.
Combine \eqref{eq1} and \eqref{eq2} to obtain \eqref{eq10}.
For Cram\'er's inequality \eqref{eq11} we refer to, e.g., \citet{MR132852}.
For $x=0$ and $\nu\in\N_0$ we have
\begin{equation}\label{eq3}
|h_{2\nu}(x)| = \sqrt{(2\nu)!} / (\nu! \cdot 2^\nu),
\end{equation}
see, e.g., \citet[Eqn.~(5.5.5)]{S75}.
Furthermore, \eqref{eq2} implies 
\begin{equation}\label{eq4}
\frac{\nu! \cdot 2^{\nu}}{\sqrt{(2\nu)!}}
\approx
\pi^{1/4}\cdot \nu^{1/4}
\end{equation}
as $\nu\to\infty$.
Combine \eqref{eq3} and \eqref{eq4} to obtain \eqref{eq12}.
\end{proof}

We will also employ the following estimate (with a suboptimal constant).

\begin{lemma}\label{Lemma_3.1b}
For every $\nu\in \N_0$ and every $x\in \R$ we have
\[
|h_\nu(x)| \le 2^\nu \max( 1, |x|^\nu).
\]
\end{lemma}

\begin{proof}
Due to \eqref{eq:rep_hermite_polynomial} we have
\[
|h_\nu(x)| \le \sum^{\lfloor \nu/2 \rfloor}_{k=0}  
\frac{\sqrt{\nu !}}{2^k\, k! \, (\nu -2k)! } \cdot \max(1, |x|^\nu).
\]
Furthermore, the elementary estimate $(2k)! \le 2^{2k} (k!)^2$ 
implies 
$ \frac{\sqrt{\nu !}}{2^k\, k! } \le  \frac{\nu !}{(2k)! }$
for $k \le \nu/2$, and therefore 
\[
\sum^{\lfloor \nu/2 \rfloor}_{k=0}   
\frac{\sqrt{\nu !}}{2^k\, k! \, (\nu -2k)! }
\le 
\sum^{\lfloor \nu/2 \rfloor}_{k=0} \frac{\nu !}{(2k)! \, (\nu -2k)! }
\le
\sum^{\nu}_{k=0} \frac{\nu !}{k! \, (\nu -k)! } =2^\nu.
\qedhere
\]
\end{proof}

With $e_\nu = h_\nu$ and Fourier weights $\alpha_\nu \in {]0,\infty[}$ 
for $\nu \in \N_0$, the summability property \eqref{g23} reads as
\begin{equation}\label{g49}
\sum_{\nu \in \N_0} \alpha_\nu^{-1} \cdot |h_\nu(x)|^2 < \infty.
\end{equation}
It is closely related to the condition
\begin{equation}\label{g41}
\sum_{\nu \in \N} \alpha_\nu^{-1} \cdot \nu^{-1/2} < \infty.
\end{equation}

\begin{lemma}\label{l42}
If \eqref{g41} is satisfied,
then \eqref{g49} holds for every $x \in \R$.
\end{lemma}

\begin{proof}
The statement follows directly from \eqref{eq10}. 
\end{proof} 

Due to \eqref{eq12}
the reverse implication holds true under mild additional assumptions. For 
instance, if \eqref{g49} holds for $x=0$ and
$\alpha_\nu \leq \alpha_{\nu+1}$ holds for all $\nu$ sufficiently
large, then \eqref{g41} is satisfied.

\begin{defi}
 Assume that $\inf_{\nu \in \N_0} \alpha_\nu >0$ and that \eqref{g49} is satisfied for every $x \in \R$. The reproducing kernel 
\begin{equation}\label{g54}
\phantom{\qquad\quad x,y \in \R,}
k(x,y) := 
\sum_{\nu \in \N_0} \alpha_{\nu}^{-1} \cdot h_\nu(x) \cdot
h_\nu(y),
\qquad\quad x,y \in \R,
\end{equation}
is called a \emph{Hermite kernel} and the Hilbert space $H(k)$ is called a \emph{Hermite space} of functions of a single real variable.
\end{defi}

Observe that Proposition~\ref{l24} provides a characterization 
of the Hermite space $H(k)$ 
as a linear subspace of $L^2(\mu_0)$.
For definition and terminology we refer to
\citet[Sec.~3.1]{IL15}
and \citet[Sec.~2]{DILP18}, 
where $\sum_{\nu \in \N} \alpha_\nu^{-1} < \infty$ or
$\sum_{\nu \in \N} \alpha_\nu \cdot \nu^{-1/6} < \infty$, respectively, 
is assumed, cf. \eqref{g41} and Lemma \ref{l42}.

\begin{exmp}\label{r20}
Let $r> 0$.
If $\alpha_\nu := (\nu+1)^r$ for $\nu \in \N_0$, then
\eqref{g41} is satisfied if and only if $r > 1/2$. 
Moreover, the spaces $H(k)$ with $r > 1/2$ form an
interpolation scale, see Remark~\ref{r30}.
In the particular 
case $r \in \N$ the space $H(k)$ is, up to equivalence of norms, the 
Sobolev space of all continuous
functions in $L^2(\mu_0)$ with weak derivatives up to order $r$
belonging to $L^2(\mu_0)$,
see, e.g., \citet[Prop.~1.5.4]{MR1642391} 
and \citet[Sec.~2]{DILP18}.
Unfortunately, it seems to be unknown if there is a corresponding 
natural interpolation scale of weighted Sobolev spaces of fractional 
order. The known interpolation results for weighted Sobolev spaces 
of functions on $\R$ require weights that decrease slower
at infinity than the Gaussian weight function, see 
\citet[Sec.~3.3.1]{MR503903}. 
Hence an intrinsic characterization of the functions in $H(k)$ for 
non-integer $r$ via fractional smoothness remains open.

Necessary as well as sufficient conditions, involving
classical differentiability, for $f \in H(k)$ to hold true are
presented in \citet[Sec.~3.1.1]{IL15} in the case $r > 1$.
\end{exmp}

\begin{exmp}\label{r22}
Let $r,b > 0$.
If $\alpha_0:=1$ and $\alpha_\nu := 2^{r \cdot \nu^b}$ for $\nu \geq 1$, 
then \eqref{g41} is satisfied in any case, and $H(k)$ consists of
$C^\infty$-functions. 
Moreover, in the case $b=1$ the spaces $H(k)$ with $r > 0$ form an 
interpolation scale, see Remark~\ref{r30}.
If $b \geq 1$ then the elements of $H(k)$ are real analytic
functions, see \citet[Prop.~3.7]{IL15} for the case $b=1$, which
trivially extends to any $b \geq 1$.
It is also observed in \citet[Rem.~1]{IKPW16b} without proof that, for 
$0<b<1$, the elements of $H(k)$ belong to the Gevrey class
$G^\sigma(\R)$ of index $\sigma := 1/b$.
By definition, $G^\sigma(\R)$ with
index $\sigma \ge 1$ consists of all infinitely differentiable functions 
$f$  on $\R$ such that for any $R>0$ there exists $c>0$ with
\[
 |f^{(\ell)} (x)| \le (c^{\ell +1} \ell!)^{\sigma}
\]
for $\ell \in \N_0$ and $x\in \left[-R,R\right]$.
The Gevrey class $G^1(\R)$ coincides with the class of real analytic 
functions on $\R$, while $G^\sigma(\R)$ contains compactly supported 
functions $f \neq 0$ for $\sigma>1$, see \citet{R93}. 
The following lemma improves upon the known regularity results for 
functions in $H(k)$.
\end{exmp}

\begin{lemma}\label{l76}
Let $k$ be the Hermite kernel specified in Example~\ref{r22}
and let 
\[
\sigma := \max( 1 , 1/2b ).
\]
 Then 
 \[
  H(k) \subseteq G^\sigma(\R).
 \]
 In particular, $H(k)$ contains only real analytic functions if $b\ge 1/2$.
\end{lemma}
\begin{proof}
The first part of the proof follows the argument of \citet[Prop.~3.7]{IL15} and is  
based on the formula
\[
h_\nu^{(\ell)} = \sqrt{\frac{\nu!}{(\nu-\ell)!}} h_{\nu-\ell}
\] 
for $\ell \leq \nu$
that allows it to directly relate the coefficients in the orthogonal 
expansion of $f^{(\ell)}$ with respect to the Hermite polynomials
with the coefficients in the expansion of $f \in H(k)$.
Then Cram\'er's inequality \eqref{eq11} together with the Cauchy-Schwarz 
inequality lead to the pointwise estimate
\[
|f^{(\ell)} (x)| \le \exp(x^2/4) \cdot \|f\|_{H(k)} \cdot
\left( 
\sum_{\nu=\ell}^\infty \frac{\nu!}{(\nu-\ell)!} \alpha_\nu^{-1} 
\right)^{1/2}.
\]
It follows that $f$
belongs to the Gevrey class $G^\sigma(\R)$ with 
index $\sigma  \ge 1$ if there exists $c>0$ such that
\begin{equation} \label{eq:prooflemgev}
\sum_{\nu=\ell}^\infty \frac{\nu!}{(\nu-\ell)!} \alpha_\nu^{-1} 
\leq (c^{\ell +1} \ell!)^{2 \sigma}
\end{equation}
for $\ell \in \N_0$.
We now use 
\[
 \alpha_\nu^{-1} = 2^{- r \cdot \nu^b} \le 2^r \cdot 2^{-r \cdot t^b}
\]
for $t \in [ \nu, \nu+1]$ to estimate
\[
 \sum_{\nu=\ell}^\infty \frac{\nu!}{(\nu-\ell)!} \alpha_\nu^{-1} 
 \leq
 \sum_{\nu=\ell}^\infty \nu^\ell \alpha_\nu^{-1} 
 \leq
 2^r \int_\ell^\infty t^\ell 2^{-r \cdot t^b} \, d t 
 \leq c^{\ell +1} \int_0^\infty u^{(\ell+1)/b-1} e^{-u} \, d u
\]
with $c>0$ depending only on $r$ and $b$.
The last integral is just $\Gamma \left( \frac{\ell+1}{b}\right)$. Now using Stirling's formula both as upper bound for 
$\Gamma \left( \frac{\ell+1}{b}\right)$ and as lower bound for $\ell !$ shows that there exists $c>0$ depending only on $r$ and $b$ such that 
\[
 \sum_{\nu=\ell}^\infty \frac{\nu!}{(\nu-\ell)!} \alpha_\nu^{-1} 
 \leq
  ( c^{\ell +1} \ell!)^{1/b}.
\]
Now comparing with \eqref{eq:prooflemgev} proves the lemma.
\end{proof}

\subsection{Functions of Infinitely Many Variables}\label{s3.2}
\label{SUBSEC:Hermite_inf}

Our construction of Hermite spaces of functions of infinitely many
real variables is based on Fourier weights 
$\alpha_{\nu,j} \in \left]0,\infty\right[$ for $\nu \in \N_0$ and 
$j \in \N$ with the following properties:
\begin{itemize}
\item[(\ha)]
$\sum_{\nu \in \N} \alpha_{\nu,j}^{-1} \cdot \nu^{-1/2} < \infty$
for every $j \in \N$
and
$\alpha_{\nu,j} \leq \alpha_{\nu+1,j}$ for all $\nu,j \in \N$,
\item[(\hb)]
$\sum_{j \in \N} |\alpha_{0,j} - 1| < \infty$,
\item[(\hc)]
there exists an integer $j_0 \in \N$ such that
$\sum_{\nu \in \N,\ j \geq j_0} \alpha_{\nu,j}^{-1} < \infty$.
\end{itemize}

If (\ha) is satisfied then
we may consider the 
Hermite kernels 
\[
\phantom{\qquad\quad x,y \in \R,}
k_j (x,y) := 
\sum_{\nu\in \N_0} 
\alpha_{\nu,j}^{-1} \cdot h_\nu (x) \cdot h_\nu (y),
\qquad\quad x,y \in \R,
\]
as well as the Hermite
spaces $H(k_j)$ as linear subspaces of $L^2(\mu_0)$,
see Proposition~\ref{l24} and Lemma~\ref{l42}. Observe that
\[
k_j^* := k_j - \alpha_{0,j}^{-1}
\]
is a reproducing kernel, too, which follows from $h_0=1$.
Cf.\ Remark~\ref{r42}.
Property (\hb) provides a certain amount of flexibility
for the choice of the Fourier weights $\alpha_{0,j}$;
values different from $1$ naturally arise in Section~\ref{s:iso}.

In each of the spaces $H(k_j)$ we choose the constant function  
$\alpha_{0,j}^{-1/2}$ as unit vector
and we study the corresponding incomplete tensor product
\begin{equation}\label{g66}
H^{(\ba_0^{-1/2})} := \bigotimes_{j \in \N} \left(H(k_j)\right)^{(\alpha_{0,j}^{-1/2})},
\end{equation}
see Definition \ref{def:incprod}.

We will employ Theorem~\ref{ta1} to identify the space $H^{(\ba_0^{-1/2})}$ with
a reproducing kernel Hilbert space of functions on a suitable
subset $\X$ of $\R^\N$.
For any non-empty domain $\X \subseteq \R^\N$ that ensures the 
convergence of the partial products $\prod_{j=1}^J k_j(x_j,y_j)$ as 
$J \to \infty$ we consider the tensor product kernel $K$
given by
\begin{equation}\label{g27}
\phantom{\qquad\quad \bx,\by \in \X.}
K(\bx,\by) := \prod_{j=1}^\infty k_j(x_j,y_j),
\qquad\quad \bx,\by \in \X.
\end{equation}

For a finite subset $\fu$ of $\N$, we put 
\[
c_{\fu} :=  
\prod_{j \in \fu} \alpha_{0,j}
\cdot \prod_{j=1}^\infty \alpha_{0,j}^{-1},
\]
as well as
\[
\phantom{\qquad\quad \bx,\by \in \X.}
k^\ast_{\fu} (\bx,\by) := 
\prod_{j \in \fu} 
k^*_j(x_j,y_j),
\qquad\quad \bx,\by \in \X.
\]

\begin{lemma}\label{l51}
Assume that {\rm (\ha)} and {\rm (\hb)} are satisfied.
Then the maximal domain for $K$ is given by
\begin{equation}\label{g30a}
\X :=  \Bigl\{ \bx \in \R^\N \colon \sum_{j\in \N} k_j^* (x_j,x_j) <
\infty \Bigr\}.
\end{equation}
Moreover, 
\begin{equation}\label{g30}
\phantom{\qquad\quad \bx,\by \in \X,}
K(\bx,\by) = 
\sum_{\fu \in \bU} c_\fu \cdot k^\ast_\fu(\bx,\by),
\qquad\quad \bx,\by \in \X,
\end{equation}
with absolute convergence.
\end{lemma}

\begin{proof}
At first, we consider the particular case that $\alpha_{0,j}=1$
for every $j \in \N$.
For $J \in \N$ let $\bUJ$ denote the power set of 
$\{1,\dots,J\}$.
For $\bx,\by \in \R^\N$ we have
\[
\prod_{j=1}^J k_j(x_j,y_j) =
\prod_{j=1}^J (1+ 
k^*_j(x_j,y_j)).
\]
Since 
\[
\ln \left( \prod_{j=1}^J (1+ k^*_j(x_j,x_j)) \right) = 
\sum_{j=1}^J \ln (1+ k^*_j(x_j,x_j))
\]
and $t/2 \le \ln(1+t) \le t$ for all $0\le t\le 1$, we conclude that
$\prod_{j=1}^J k_j(x_j,x_j)$ converges as $J \to \infty$
if and only if $\bx \in \X$.
Let $\bx,\by \in \X$. 
Using
\begin{equation}\label{g44}
|k^*_j(x_j,y_j)| \leq \sqrt{k^*_j(x_j,x_j)}\cdot \sqrt{k_j^*(y_j,y_j)}
\end{equation}
and the Cauchy Schwarz inequality for sums, we obtain
\begin{equation*}
\begin{split}
\sum_{\fu \in \bUJ} \prod_{j \in \fu} |k^*_j(x_j,y_j)|
&\leq \left( 
\sum_{\fu \in \bUJ} \prod_{j \in \fu} k^*_j(x_j,x_j) \right)^{1/2} \cdot
\left( 
\sum_{\fu \in \bUJ} \prod_{j \in \fu} k^*_j(y_j,y_j) \right)^{1/2}\\
&=
\left( \prod_{j=1}^J (1+  k^*_j(x_j,x_j)) \right)^{1/2} \cdot
\left( \prod_{j=1}^J (1+  k^*_j(y_j,y_j)) \right)^{1/2}.
\end{split}
\end{equation*}
Consequently,
\[
\prod_{j=1}^J (1+  k^*_j(x_j,y_j))
 =  \sum_{\fu \in \bUJ} \prod_{j \in \fu} 
k^*_j(x_j,y_j)
\]
is absolutely convergent as $J \to \infty$.
We conclude that
$\X$ is the maximal domain for $K$, and 
\eqref{g30} is satisfied with absolute convergence.

The general case is easily reduced to the particular case,
since 
$\prod^J_{j=1} \alpha_{0,j}$ converges and 
$\lim_{j \to \infty} \alpha_{0,j}=1$ due to (\hb) and since
\begin{equation}\label{g51}
k_j = \alpha_{0,j}^{-1} \cdot (1 + \alpha_{0,j} \cdot k_j^*).
\qedhere
\end{equation}
\end{proof}

Observe that $\X$ does not depend on the Fourier weights
$\alpha_{0,j}$.

\begin{defi}
 Assume that {\rm (\ha)} and {\rm (\hb)} are satisfied. The reproducing kernel  $K$ defined by \eqref{g27} on the maximal domain 
$\X$ given by \eqref{g30a}
is called a \emph{Hermite kernel} and the Hilbert space $H(K)$ is called a \emph{Hermite space} of functions of  infinitely many real variables.
\end{defi}

Let $\bx,\by \in \X$. It follows from \eqref{g44} that
$\sum_{j \in \N} |k_j^*(x_j,y_j)| < \infty$. 
If (\ha) and (\hb) are satisfied, then
the product in \eqref{g27} is convergent in the sense of the definition 
in Section~\ref{as1},
and $K(\bx,\by)=0$ is equivalent to the existence of $j \in \N$
with $k_j(x_j,y_j)=0$.

The Hermite space $H(K)$ is, in the sense of 
Theorem~\ref{ta1}, the incomplete tensor product $H^{(\ba_0^{-1/2})}$ of the 
Hermite spaces $H(k_j)$ based on the constant functions 
$\alpha_{0,j}^{-1/2}$ for $j \in \N$ as unit vectors as in \eqref{g66}. 

We now analyze the maximal domain $\X$ in more detail. 
Put $\ba_\nu^{-1} := (\alpha_{\nu,j}^{-1})_{j \in \N}$ for $\nu \in \N$,
and observe that (\hc) implies $\sum_{j \in \N}
\alpha_{\nu,j}^{-1} < \infty$.
The space $\ell^{2\nu}(\ba_\nu^{-1})$ consists of all sequences 
$\bx \in \R^\N$ such that
\[
\sum_{j \in \N}  \alpha_{\nu,j}^{-1} \cdot x_j^{2\nu} < \infty.
\]
In the sequel, we consider the countable product $\mu$ of the standard 
normal distribution $\mu_0$ on $\R^\N$.

\begin{prop}\label{lem:domain2}
Assume that {\rm (\ha)} is satisfied.
Then we have 
\[
 \X \subseteq \bigcap_{\nu\in\N} \ell^{2\nu}(\ba_\nu^{-1})
\subsetneq \R^\N.
\]
Assume that {\rm (\ha)} and {\rm(\hc)} are satisfied.
Then we have
\[
\mu(\X)=1 
\]
and
\[
\ell^\infty \subsetneq \X. 
\]
\end{prop}

\begin{proof}
Put $\ell^{2\nu} := \ell^{2\nu}(\ba_\nu^{-1})$, 
and observe that already $\ell^2 \subsetneq \R^\N$.
We prove the remaining part of the first statement by induction.

Since $h_1(x_j) = x_j$ and 
$\alpha_{1,j}^{-1} \cdot |h_1(x_j)|^2 \leq k^{*}_j(x_j,x_j)$, we
obtain $\X \subseteq \ell^2$.
Let $\bx \in \X$, let $\nu > 1$, and
assume that $\bx \in \ell^{2 \kappa}$ holds for all
$1 \leq \kappa < \nu$. 
Since $h_\nu$ is either an even or an odd polynomial, $h_\nu^2$ is 
an even polynomial of degree $2 \nu$ and can thus be 
written as
\[
h^2_\nu(x) = \sum^\nu_{\kappa=0} \beta_\kappa \cdot x^{2\kappa}
\]
with suitable $\beta_\kappa \in \R$, where $\beta_\nu \neq 0$.
Consequently,
\[
x^{2\nu} = 1/\beta_\nu \cdot \left( |h_\nu(x)|^2 -
\sum_{\kappa=0}^{\nu-1} \beta_\kappa \cdot x^{2 \kappa}\right).
\]
Due to the second condition in (\ha) and our induction hypothesis 
we have for all $1 \leq \kappa < \nu$
\[
\sum_{j\in\N} \alpha^{-1}_{\nu,j} \cdot  x^{2\kappa}_j 
\le \sum_{j\in\N} \alpha^{-1}_{\kappa,j} \cdot  x^{2\kappa}_j 
< \infty.
\]
Furthermore,
\begin{equation*}
\sum_{j \in \N}  \alpha_{\nu,j}^{-1} \cdot |h_\nu(x_j)|^2 
\le \sum_{ j\in\N} k^*_j(x_j,x_j)
< \infty.
\end{equation*}
Altogether this yields
\[
\sum_{j\in \N} \alpha^{-1}_{\nu, j}  \cdot x^{2\nu}_j < \infty,
\]
implying that $\bx \in \ell^{2\nu}$,
which proves the first statement.

Since
\[
\int_{\R^\N} k_j^*(x_j,x_j) \, d \mu(\bx) = \sum_{\nu \in \N}
\alpha_{\nu,j}^{-1},
\]
we get 
\[
\int_{\R^\N} \sum_{j \geq j_0}k_j^*(x_j,x_j) \, d \mu(\bx) < \infty
\]
from (\hc) using the monotone convergence theorem. Hence
$\sum_{j \in \N} k_j^*(x_j,x_j) < \infty$ holds $\mu$-a.e., i.e.,
$\mu(\X)=1$. 
Combine Cram\'er's inequality, see Lemma~\ref{l41},
and (\hc) to conclude that $\ell^\infty \subseteq \X$. 
Finally, $\ell^\infty \neq \X$, since $\mu(\ell^\infty)=0$.
\end{proof}

Recall that $\bN$ denotes the set of all sequences 
$\bn := (\nu_j)_{j \in \N}$ in $\N_0$ such that 
$\{j \in \N \colon \nu_j \neq 0\}$ is finite.
For $\bn \in \bN$ we set
\[
h_\bn (\bx) := \prod_{j=1}^\infty h_{\nu_j}(x_j)
\]
for $\bx \in \X$, which is well-defined since $h_0=1$. Moreover,
we set
\[
\alpha_\bn := \prod_{j=1}^\infty \alpha_{\nu_j,j},
\]
which is well-defined if (\hb) is assumed.

In addition to the original Fourier weights we also consider
\[
\alpha_{\nu,j}^\prime :=
\begin{cases}
1, & \text{if $\nu=0$,}\\
\alpha_{\nu,j}, & \text{otherwise,}
\end{cases}
\]
and we use $\alpha^\prime_\bn$, $K^\prime$, and $k_j^\prime$ to denote the
corresponding products and the corresponding reproducing kernels,
respectively. 
The counterpart to (\hb) is trivially satisfied
for the new Fourier weights, and the counterparts to 
(\ha) and (\hc)
follow from the respective properties of the 
original Fourier weights.

\begin{lemma}\label{l44}
Assume that {\rm (\hb)} is satisfied. Then we have
\[
c_{\min} := \inf_{\bn \in \bN} \frac{\alpha_\bn}{\alpha^\prime_\bn} =
\prod_{j=1}^\infty \min (\alpha_{0,j},1) > 0
\]
and
\[
c_{\max} := \sup_{\bn \in \bN} \frac{\alpha_\bn}{\alpha^\prime_\bn} =
\prod_{j=1}^\infty \max (\alpha_{0,j},1) < \infty. 
\]
\end{lemma}

\begin{proof}
Use (\hb) to conclude that
\[
\inf_{\bn \in \bN} \frac{\alpha_\bn}{\alpha^\prime_\bn} =
\lim_{J \to \infty} \biggl(\prod_{j=1}^J \min(\alpha_{0,j},1) \cdot
\prod_{j=J+1}^\infty \alpha_{0,j} \biggr)=
\prod_{j=1}^\infty \min (\alpha_{0,j},1) > 0.
\]
The second statement is verified analogously.
\end{proof}

Remark~\ref{onb} and  Theorem \ref{ta2} show that the functions $h_\bn$ 
with $\bn \in \bN$ form an orthonormal basis of 
$L^2(\mu) := L^2(\R^\N,\mu)$, which may obviously be identified 
with $L^2(\X,\mu)$ if $\mu(\X)=1$.
Together with the following lemma,
Proposition~\ref{l24} provides a 
characterization of the Hilbert space $H(K)$ as a linear subspace of
$L^2(\mu)$.

\begin{lemma}\label{l77}
Assume that 
{\rm (\ha)}--{\rm (\hc)} are satisfied.
Then we have 
\[
\inf_{\bn \in \bN} \alpha_\bn > 0
\]
and
\begin{equation}\label{g22}
\sum_{\bn \in \bN} \alpha_\bn^{-1} \cdot |h_\bn(\bx)|^2 < \infty
\end{equation}
for $\bx \in \X$
as well as
\begin{equation}\label{g22a}
K(\bx, \by) = \sum_{\bn \in \bN} \alpha_{\bn}^{-1} 
\cdot h_{\bn}(\bx) \cdot h_{\bn}(\by)
\end{equation}
for $\bx,\by \in \X$.
\end{lemma}

\begin{proof}
Obviously, (\hc) implies $\sum_{j \geq j_0} \alpha_{1,j}^{-1} <
\infty$. We use the second condition in (\ha) to conclude that 
there exists an integer $j_1 \geq j_0$ with 
\[
\inf_{j \geq j_1} \inf_{\nu \in \N} \alpha_{\nu,j} =
\inf_{j \geq j_1} \alpha_{1,j} 
\geq 1.
\]
Moreover, 
\[
c := \inf_{j=1,\dots,j_1-1} \inf_{\nu \in \N_0} \alpha_{\nu,j} > 0,
\] 
which follows from (\ha). Together with (\hb),
cf.\ Lemma \ref{l44}, this implies
\[
\inf_{\bn \in \bN}
\alpha_{\bn} \geq c^{j_1-1} \cdot \prod_{j=j_1}^\infty 
\min(\alpha_{0,j},1) > 0.
\]

Consider the particular case that $\alpha_{0,j}=1$
for every $j \in \N$. Let $\bx \in \R^\N$. 
According to \citet[Lem.~B.1]{GHHRW2020}, applied with 
$\beta_{\nu,j} := \alpha_{\nu,j}^{-1} \cdot |h_\nu(x_j)|^2$,
\eqref{g22} holds if and only if 
\[
\sum_{\nu,j \in \N} \alpha_{\nu,j}^{-1} \cdot |h_\nu(x_j)|^2 < \infty,
\]
and the latter is equivalent to $\bx \in \X$.
Let $\bx,\by \in \X$.
Applying \citet[Lem.~B.1]{GHHRW2020} again, this time with 
$\beta_{\nu,j} := \alpha_{\nu,j}^{-1} \cdot h_\nu(x_j) \cdot
h_\nu(y_j)$, we obtain \eqref{g22a}.

The general case is easily reduced to the particular case
by means of Lemma \ref{l44} and \eqref{g51}.
\end{proof}

\begin{lemma}\label{l43}
Assume that {\rm (\ha)} is satisfied. Then we have 
$H(k_j) = H(k_j^\prime)$ for $j\in\N$
as vector spaces, and the corresponding identity maps have norms
\begin{equation*}
 \| H(k_j) \hookrightarrow H(k_j^\prime) \| = 
\min(\alpha_{0,j},1)^{-1/2} \quad \text{and} \quad 
\| H(k_j^\prime) \hookrightarrow H(k_j) \| =  \max(\alpha_{0,j},1)^{1/2}.
\end{equation*}
Assume that {\rm (\ha)}--{\rm (\hc)} are satisfied.
Then we have $H(K) = H(K^\prime)$
as vector spaces, and the corresponding identity maps have norms
\begin{equation*}
\| H(K) \hookrightarrow H(K^\prime) \| = c_{\min}^{-1/2} 
\quad \text{and} \quad
\| H(K^\prime) \hookrightarrow H(K) \| =  c_{\max}^{1/2}. \\
\end{equation*}
\end{lemma}

\begin{proof}
Let (\ha) hold, and let $j\in \N$. 
Proposition~\ref{l24} shows that $H(k_j) = H(k_j^\prime)$ and 
confirms the claimed values of the norms of the embeddings.

Let now {\rm (\ha)}--{\rm (\hc)} be satisfied.
We combine Proposition~\ref{l24}, Lemma \ref{l44},
and Lemma \ref{l77} to obtain $H(K^\prime) = H(K)$ as well as the 
claimed values of the norms of the embeddings.
\end{proof}

\begin{rem}\label{rem:ANOVA}
Assume that {\rm (\ha)}--{\rm (\hc)} are satisfied.
For each $j\in\N$ we have
\[
\int_{\R} k^\ast_j(x,y) \, d \mu_0(y) = 0
\]
for all $x\in \R$, cf.\ Remark~\ref{r42}. This property 
induces a function space decomposition of the tensor product 
space $H(K)$ of ANOVA-type, cf.\ \citet[Rem.~2.12]{BG12}. Let us make 
this more precise:
Due to \eqref{g30} we may write $H(K)$ as orthogonal sum
\[
H(K) = \bigoplus_{\fu\in \bU} H(c_{\fu}\cdot k^\ast_{\fu}).
\]
Note that due to $c_{\fu} >0$ we have 
$H(c_{\fu}\cdot k^\ast_{\fu}) = H(k^\ast_{\fu})$
as vector spaces for all $\fu\in\bU$.
Hence each function $f\in H(K)$ can be uniquely represented as
\[
f = \sum_{\fu\in\bU} f_{\fu},
\]
where $f_{\fu} \in H(k^\ast_{\fu})$ for each $\fu\in\bU$,
and its norm is given by
\[
\|f\|^2_{H(K)} = 
\sum_{\fu\in\bU} \|f_{\fu} \|^2_{H(c_{\fu}\cdot k^\ast_{\fu})}
= \sum_{\fu\in\bU}c_{\fu}^{-1} \cdot \|f_{\fu} \|^2_{H(k^\ast_{\fu})}.
\]
This function decomposition is the infinite-dimensional 
\emph{ANOVA-decomposition} of $f$. 
More precisely, we obtain 
\[
f_{\emptyset} = \int_{\X} f\, d\mu
\]
and, recursively, for arbitrary $\emptyset \neq \fu \in \bU$ 
\[
\phantom{\qquad\quad \bx\in \X,}
f_{\fu}(\bx) = 
\int_{\R^{\N\setminus \fu}} f(\bx_{\fu}, \by_{\N \setminus \fu}) 
\, d \mu_0^{\N \setminus \fu}(\by_{\N \setminus \fu}) - 
\sum_{\fv \subsetneq \fu} f_{\fv}(\bx),
\qquad\quad \bx\in \X,
\]
and 
\[
\|f\|^2_{L^2(\mu)} = 
\sum_{\fu\in\bU} \|f_{\fu} \|^2_{L^2(\mu)}
\]
as well as
\[
\sigma^2(f) = \sum_{\emptyset \neq \fu\in\bU} \sigma^2(f_{\fu}), 
\]
where $\sigma^2(\cdot)$ denotes the variance.
In particular, for $f \in H(K)$ the ANOVA-components $f_{\fu}$
with $\fu\in\bU$ are orthogonal in $H(K)$ as well as in $L^2(\mu)$, 
and each function $f_{\fu}$ depends only on the variables $x_j$
with $j\in \fu$.
\end{rem}

\begin{rem}\label{r21}
Infinite tensor products $K$ of reproducing kernels $k_j \colon D
\times D \to \R$ provide a convenient setting for the study of
computational problems for functions of infinitely many variables,
and most often the whole Cartesian product $D^\N$ is the natural
domain for the kernel $K$. 
The complexity of problems of this kind has 
first been analyzed in 
\citet{HicWan2001}, \citet{HicEtAl2010}, \citet{KuoEtAl10},
and we refer to, e.g., \citet{GHHRW2020}
for recent results and references. 
So far two different types of spaces of functions of infinitely many 
variables have been studied: First, \emph{weighted spaces}, where the 
weights model the importance of different groups of variables. This 
kind of weights was first introduced in \citet{SW98}.
Second, \emph{spaces of increasing smoothness} based on Fourier weights, 
which is the approach we follow in this paper. The effect of increasing 
smoothness on the tractability of multivariate problems was first studied 
in \citet{PapWoz10}.

Assume that $\alpha_{0,j}=1$ for every $j \in \N$.
Due to the representation \eqref{g30} the kernel $K$ is the superposition 
of finite tensor products of reproducing kernels $k_j^*$. 
In such a setting the 
\emph{multivariate decomposition method}
has been established as a powerful generic algorithm for
integration and $L^2$-approximation, if additionally
all of the $k_j^*$ are weighted anchored kernels,
which leads to a so-called anchored function space decomposition.
By definition, $k_j^*$ is \emph{anchored} at $x \in D$ if 
\[
k_j^*(x,x) = 0.
\]
We refer to the pioneering papers \citet{KuoEtAl10} and \citet{PW11}, 
and to \citet{GilEtAl2018} for a recent contribution.

Observe, however, that in the present setting of Hermite spaces
none of the kernels $k_j^*$ is anchored at any point $x \in \R$, 
since $h_1(x) = x$ for every $x \in \R$ and $h_2(0) \neq 0$.
As explained in Remark~\ref{rem:ANOVA}, instead of having an
anchored function space decomposition, $H(K)$ has an
ANOVA-decomposition. This prohibits the 
application of the standard error analysis of the
multivariate decomposition method developed in \citet{PW11} for 
integration and in \citet{Was12} for $L^2$-approximation and of standard
arguments to derive lower bounds for deterministic algorithms
on $H(K)$.

In the ANOVA setting the multivariate
decomposition method has been analyzed directly in \citet{DG13} 
for integration in the randomized setting 
and
indirectly via suitable embedding theorems in \citet{GneEtAl16,
GHHRW2020} for integration and $L^2$-approximation in the randomized and
the deterministic setting.

While \citet{DG13} and \citet{GneEtAl16} treat weighted spaces $H(K)$ 
and allow for domains $\X$ of $K$ that are proper subsets 
of $D^\N$,  the paper 
\citet{GHHRW2020} studies spaces of increasing smoothness whose kernels 
are defined on the whole Cartesian product $D^\N$.  It is
important to note that Hermite spaces $H(K)$ are never spaces of 
functions on the domain $\R^\N$, see Proposition~\ref{lem:domain2}.
In a forthcoming paper we will develop the approach
from \citet{GHHRW2020} further and apply it, in particular,
to Hermite spaces $H(K)$.
\end{rem}

\subsection{Two Examples}\label{s3.3}

In the sequel we consider two particular kinds of Fourier weights 
$\alpha_{\nu,j}$ and the corresponding Hermite spaces
$H(k_j)$ and $H(K)$, cf.\ Examples \ref{r20} and \ref{r22}. 

\begin{itemize}
\item[(PG)]
Let $r_j > 1/2$ for $j \in \N$ such that
\begin{equation}\label{g61}
\sum_{j \in \N} 2^{-r_j} < \infty.
\end{equation}
The Fourier weights with a \emph{polynomial growth} are given by 
\[
\alpha_{\nu,j} := (\nu+1)^{r_j}
\]
for $\nu,j \in \N$ and 
by any choice of
$\alpha_{0,j}$ satisfying (\hb).
\item[(EG)]
Let $r_j > 0$ for $j \in \N$ such that \eqref{g61} is
satisfied. Moreover, let $b_j > 0$ such that 
\[
\inf_{j \in \N} b_j > 0.
\]
The Fourier weights with a \emph{(sub-)exponential growth} are given by 
\[
\alpha_{\nu,j} := 2^{r_j \cdot \nu^{b_j}}
\]
for $\nu,j \in \N$ and 
by any choice of
$\alpha_{0,j}$ satisfying (\hb).
\end{itemize}

Finite tensor products of Hermite spaces with different kinds
of Fourier weights have been studied, e.g., in the following
papers.
Fourier weights with a polynomial growth are considered in 
\citet{IL15} for $r_j > 1$ and in \citet{DILP18} for $r_j \in \N$.
For Fourier weights with a (sub-)exponential growth we refer to
\citet{IL15}, \citet{IKLP15} as well as \citet{IKPW16b}, 
and \citet{IKPW16a}
for the cases $b_j = 1$, $b_j \geq 1$, and $\inf_{j \in \N}
b_j > 0$, respectively.

\begin{rem}\label{r50}
Let
\[
\hat{r} := \liminf_{j \to \infty} \frac{r_j}{\ln(j)}.
\]
Then $\hat{r} > 1/\ln(2)$ is a sufficient, and $\hat{r} \geq 1/\ln(2)$
is a necessary condition for \eqref{g61} to hold.
See, e.g., \citet[Lem.~B.3]{GHHRW2020}.
\end{rem}

\begin{lemma}\label{l2}
In both cases, {\rm (PG)} and {\rm (EG)}, we have 
\rm{(\ha)}--\rm{(\hc)}.
\end{lemma}

\begin{proof}
Due to the considerations in Examples \ref{r20} and \ref{r22}, 
respectively, (\ha) is satisfied in both cases, and
(\hb) is satisfied by assumption.

For proving (\hc) we use
\[
\sum_{k \geq k_0} k^{-\tau}
\leq k_0^{-\tau} + \int_{k_0}^\infty k^{-\tau}\, d \tau
= k_0^{-\tau} \cdot \Big(1 + \frac{k_0}{\tau-1}\Big)
\]
for $\tau > 1$ and $k_0 \in \N$.

Consider the case (PG). Here we choose $j_0 \in \N$ such that 
$r_j \geq 2$ for every $j \geq j_0$. It follows that
\[
\sum_{\nu \in \N} \alpha_{\nu,j}^{-1} =
\sum_{\nu \in \N} (\nu+1)^{-r_j} 
\leq 2^{-r_j} \cdot 3
\]
for $j \geq j_0$, which yields (\hc), as claimed.

Consider the case (EG). 
Here we put $b := \inf_{j \in \N} b_j > 0$ and
choose $j_0 \geq 2$ and 
$c>0$ such that $2^{r_j} \geq j^c$
for every $j \geq j_0$, see Remark~\ref{r50}. It follows that
\[
\alpha_{\nu,j} = 
2^{r_j \cdot \nu^{b_j}} \geq 2^{r_j \cdot \nu^{b}} \geq j^{c \cdot
\nu^{b}}
\]
for $j \geq j_0$ and $\nu \in \N$. We choose $\nu_0 \in \N$ such
that $c \cdot \nu_0^b > 1$. For $\nu \geq \nu_0$ this yields
\[
\sum_{j \geq j_0} \alpha_{\nu,j}^{-1}
\leq j_0^{-c \cdot \nu^{b}} \cdot \Big(1 + \frac{j_0}{c\cdot
\nu_0^b-1} \Big),
\]
while
\[
\sum_{j \geq j_0} \alpha_{\nu,j}^{-1}
\leq \sum_{j \geq j_0} 2^{-r_j} < \infty
\]
for $1 \leq \nu < \nu_0$. Since
\[
\sum_{\nu \in \N} j_0^{-c \cdot \nu^{b}} < \infty,
\]
we obtain (\hc), as claimed. 
\end{proof}

According to Proposition~\ref{lem:domain2} and Lemma~\ref{l2}, 
the maximal domain $\X$ satisfies
\[
\X \subseteq \bigcap_{\nu\in\N} \ell^{2\nu}(\ba_\nu^{-1}) 
\subseteq \ell^2(\ba_1^{-1})
\]
in both cases, (PG) and (EG).

\begin{prop}\label{Lemma:Char_X_EG}
In the case {\rm (EG)} with exponential growth, i.e., if 
$b_j \geq 1$ for every $j \in \N$, we have
\[
\X = \bigcap_{\nu\in\N} \ell^{2\nu}(\ba_\nu^{-1}) = \ell^2(\ba_1^{-1}).
\]
\end{prop}

\begin{proof}
It suffices to show that 
$$\ell^2(\ba_1^{-1}) \subseteq \X.$$ 
Since $\X$ becomes larger if we increase the $b_j$, $j\in \N$, and 
$\ell^2(\ba_1^{-1})$ is independent of the choice of the $b_j$, $j\in\N$, 
we actually only need to consider the case where 
$b_j=1$ for every $j \in \N$. 
Let $\bx \in \ell^2(\ba_1^{-1})$. To prove $\bx \in \X$ it suffices to 
prove
\begin{equation}\label{est_sum_j_1}
\sum_{j\ge j_1} k^*_j(x_j,x_j) <\infty 
\end{equation}
for some $j_1\in\N$.
For all $j\in \N$ we put $z_j:= \max( 1, |x_j|)$. Due to 
Lemma~\ref{Lemma_3.1b} there exists a $C>0$ such that
$
|h_\nu(x_j)|^{2} \le C^\nu |z_j|^{2\nu}
$
for all $\nu\in\N$. 
Therefore we obtain, for all $j_1 \in \N$,
\begin{align*}
\sum_{j\ge j_1} k^*_j(x_j,x_j) 
&= \sum_{j\ge j_1} \sum_{\nu \in \N} \alpha^{-1}_{\nu,j} 
\cdot |h_\nu (x_j)|^2 \\
&\le  \sum_{j\ge j_1} 
\sum_{\nu \in \N} 2^{-r_j\,\nu} C^\nu \cdot |z_j|^{2\,\nu} \\
&\le \sum_{\nu\in\N} \Bigg( C \sum_{j\ge j_1} 2^{-r_j}
\cdot |z_j|^2 \Bigg)^\nu.
\end{align*}
Since $\bx \in \ell^2(\ba_1^{-1})$ 
and \eqref{g61} is satisfied, 
we have 
$
\sum_{j\in\N} 2^{-r_j} |z_j|^2 < \infty. 
$
Hence we may choose $j_1$ to be large enough to yield
$$
\sum_{j\ge j_1} 2^{-r_j} |z_j|^2 < C^{-1}.
$$
This choice obviously ensures \eqref{est_sum_j_1}.
\end{proof}

\begin{prop}\label{p17}
Consider the case {\rm (PG)} or the case {\rm (EG)} with sub-exponential 
growth, i.e., $\limsup_{j \to \infty} b_j < 1$. 
If $\hat{r} > 1/\ln(2)$ then we have
\[
\bigcap_{\nu\in\N} \ell^{2\nu}(\ba_\nu^{-1})
\subsetneq 
\bigcap_{\nu=1}^{\nu_0} \ell^{2\nu}(\ba_\nu^{-1})
\]
for every $\nu_0 \in \N$.
\end{prop}

\begin{proof}
Since $\hat{r} > 1/\ln(2)$, we may
choose $0<\eps<1$ and $j_0\in\N$ such that
\[
2^{r_j} \geq j^{1+\eps}
\]
for $j\ge j_0$. In the case (EG) with sub-exponential growth we
may furthermore assume
\[
b_j \leq 1 - \eps
\]
for $j \geq j_0$.
Moreover, we choose $1 < a < 1+\eps$, and we put 
$\delta := 1-a/(1+\eps) > 0$. Let $\nu_0 \in \N$. For $\bx \in \R^\N$ with 
\[
x_j^{2 \nu_0} = 2^{r_j} \cdot j^{-a}
\]
we show that
\begin{equation}\label{nn1}
\bx \in \bigcap_{\nu=1}^{\nu_0} \ell^{2\nu}(\ba_\nu^{-1}),
\end{equation}
but
\begin{equation}\label{nn2}
\bx \not \in \ell^{2k\nu_0}(\ba_{k\nu_0}^{-1})
\end{equation}
for $k \in \N$ sufficiently large.

Let $j \geq j_0$ and $1 \leq \nu \leq \nu_0$.
Since $|x_j| \geq 1$ and $\alpha_{\nu,j} \geq 2^{r_j}$, we obtain
\[
\sum_{j \geq j_0} \alpha_{\nu,j}^{-1} \cdot x_j^{2\nu}
\leq
\sum_{j \geq j_0} 2^{-r_j} \cdot x_j^{2\nu_0} < \infty.
\]
Consequently, we have \eqref{nn1}.

Let $j \geq j_0$ and $k \in \N$.
Note that 
\[
x_j^{2 k \nu_0}
=
2^{r_j k} \cdot j^{-ak} 
\geq 
2^{r_j k \delta}.
\]

Consider the case (EG) with sub-exponential growth. Here we have 
\[
\alpha_{k \nu_0,j}^{-1} \geq 2^{-r_j \cdot (k\nu_0)^{1-\eps}}.
\]
For every $k \in \N$ with
\[
k \delta \geq (k \nu_0)^{1-\eps}
\]
this implies
\[
\alpha_{k \nu_0,j}^{-1} \cdot x_j^{2 k \nu_0} \geq 
2^{r_j \cdot \left(k\delta-(k\nu_0)^{1-\eps}\right)}
\geq 1,
\]
which completes the proof of \eqref{nn2} in this case.

Consider the case (PG). Here we have
\[
\alpha_{k \nu_0,j}^{-1} = (1+k\nu_0)^{-r_j}
\]
by definition. For every $k \in \N$ with
\[
2^{k\delta} \geq 1+k\nu_0
\]
this implies
\[
\alpha_{k \nu_0,j}^{-1} \cdot x_j^{2 k \nu_0} \geq 
\left( 2^{k\delta} \cdot \left(1+k\nu_0\right)^{-1}\right)^{r_j}
\geq 1,
\]
which completes the proof of \eqref{nn2} in the case (PG).
\end{proof}

In the cases (PG) and (EG) with sub-exponential growth 
we do not know if $\X = \bigcap_{\nu\in\N} \ell^{2\nu}(\ba_\nu^{-1})$
or if there is another
simple characterization of this domain.

\section{Gaussian Spaces}
\label{SEC:Gauss}

Now we study Hilbert spaces of Gaussian 
kernels, which we shortly address as Gaussian spaces. 
As before, we take $\K = \R$.

\subsection{Functions of a Single Variable}
\label{SUBSEC:Gauss_1}

At first we consider spaces of univariate functions.

\begin{defi}
Let $\sigma>0$. 
The reproducing kernel 
$\ell_\sigma$ given by 
\[
\phantom{\qquad\quad x,y \in \R,}
\ell_\sigma (x,y) :=  
\exp \left( -\sigma^2 \cdot (x-y)^2 \right),
\qquad\quad x,y \in \R,
\]
is called a \emph{Gaussian kernel} 
and the Hilbert space $H(\ell_\sigma)$ is called a \emph{Gaussian space} of functions of a single real variable with \emph{shape parameter} $\sigma$.
\end{defi}

We collect some facts about the spaces $H(\ell_\sigma)$ from 
\citet{SHS2006},
see also \citet[Sec.~4.4]{SC08}.
Each function $f\in H(\ell_\sigma)$ is the real part of 
an entire function $g$ restricted to the real line, where $g$ belongs 
to the complex reproducing kernel Hilbert space with kernel $\ell_\sigma$ 
extended to $\C$ in the obvious way. In particular, if $f$ is constant 
on any open non-empty interval, then $f$ is the zero function.
Moreover, the functions
\[
\phantom{\qquad\quad x \in \R,}
e_{\nu,\sigma}(x) := 
\frac{1}{\sqrt{\nu!}} \, \bigl(\sqrt{2} \sigma\bigr)^\nu x^\nu \cdot
\ell_\sigma(x,0),
\qquad\quad x \in \R,
\]
with $\nu \in \N_0$ form an orthonormal basis of $H(\ell_\sigma)$. 
It is easily verified that
\begin{equation}\label{g55}
\ell_\sigma(x,y) = \sum_{\nu \in \N_0} e_{\nu,\sigma} (x) \cdot
e_{\nu,\sigma} (y)
\end{equation}
with absolute convergence for all $x,y \in \R$. 
If $0<\sigma < \sigma'$, 
then $H(\ell_\sigma) \subsetneq H(\ell_{\sigma'})$ 
with a non-compact continuous embedding of norm 
$\sqrt{\sigma'/\sigma}$.
Observe that the identity \eqref{g55} for the Gaussian kernel
is reminiscent of
definition \eqref{g54} of the Hermite kernel.
Based on Mehler's formula instead 
of \eqref{g55}, a characterization of $H(\ell_\sigma)$ as an 
$L^2$-subspace in the sense of Proposition~\ref{l24} will be 
given in Section~\ref{s5.1}.

\subsection{Functions of Infinitely Many Variables}\label{s4.2}
\label{SUBSEC:Gauss_inf}

Gaussian spaces of functions of infinitely many real variables are 
based on a sequence $\bs := (\sigma_j)_{j \in \N}$ of shape parameters 
$\sigma_j > 0$;
for the corresponding Gaussian kernels and basis functions in
the univariate case we use the 
short hands $\ell_j := \ell_{\sigma_j}$ and $e_{\nu,j} := 
e_{\nu,\sigma_j}$. 

We proceed similar to Section \ref{s3.2}, but we will
encounter some important differences along the way. 
First of all, $1 \not\in H(\ell_j)$.
In each of the spaces $H(\ell_j)$ we therefore choose the unit vector
\[
v_j := e_{0,j},
\]
i.e., $v_j(x) = \exp(- \sigma_j^2 \cdot x^2)$, and we
study the corresponding incomplete tensor product
\[
G^{(\bv)} := \bigotimes_{j \in \N} \left( H(\ell_j)
\right)^{(v_j)}.
\]

We will employ Theorem \ref{ta1} again to identify the space $G^{(\bv)}$ 
with a reproducing kernel Hilbert space of functions on a suitable
subset of $\R^\N$. In the present case
we even have the convergence of the partial products 
$\prod_{j=1}^J \ell_j(x_j,y_j)$ as $J \to \infty$
for all $\bx,\by \in \R^\N$. Even more, the products $\prod_{j \in
\N} \ell_j(x_j,y_j)$ converge
in the sense of Section \ref{as1}, since $\ell_j(x_j,y_j) \in
\left]0,1\right]$. Hence we may study the tensor product kernel
$L$ given by
\begin{equation}\label{def:kernel_L}
\phantom{\qquad\quad \bx,\by \in \R^\N.}
L(\bx,\by) 
:= 
\prod_{j=1}^\infty \ell_j(x_j,y_j),
\qquad\quad \bx,\by \in \R^\N.
\end{equation}
One might expect that $G^{(\bv)}$ may be identified
with $H(L)$ in a canonical way, analogously to the result
for Hermite spaces.

For any reproducing kernel $M$ on $\R^\N \times \R^\N$ and any 
non-empty set $\Y \subseteq \R^\N$ we use $M_\Y$ to denote the 
restriction $M|_{\Y \times \Y}$.

The elementary tensors
\[
e_\bn := \bigotimes_{j\in\N} e_{\nu_j,j}
\]
with $\bn \in \bN$ form an orthonormal basis of $G^{(\bv)}$,
see Remark~\ref{onb}.
We will identify the tensor product space $G^{(\bv)}$ with 
a reproducing kernel Hilbert space $H(L_\Y)$
for a suitable non-empty domain $\Y \subseteq \R^{\N}$ via the 
linear mapping 
\[
\Phi \colon G^{(\bv)} \to \R^{\Y}, 
\]
given by 
\begin{equation}\label{def:Phi}
\Phi f(\bx) =  
\sum_{\bn\in \bN} \scp{f}{e_\bn}_{G^{(\bv)}} \cdot
\prod_{j=1}^\infty e_{\nu_j,j}(x_j)
\end{equation}
for $\bx\in \Y$, cf.\ Theorem \ref{ta1}. More precisely, $\Phi$ 
should, of course, be well-defined and should induce an isometric 
isomorphism between $G^{(\bv)}$ and $H(L_\Y)$.
It turns out, in particular, that we cannot achieve the latter
goal with $\Y = \R^{\N}$; instead, as shown below, $\Y = \X$ with
\begin{equation}\label{eq:dom_gauss}
\X := \ell^2(\bs^2)
\end{equation}
is the proper choice. By definition,
$\X$ consists of all sequences $\bx \in \R^{\N}$ such that
\[
\sum_{j\in\N} \sigma^2_j \cdot x_j^2 <\infty.
\]

Obviously,
\[
L(\bx,\by) =
\begin{cases}
\exp \bigl( -\sum_{j\in\N} \sigma_{j}^2 \cdot (x_{j}- y_j)^2\bigr),
& \text{if $\bx-\by \in \X$,}\\
0, & \text{otherwise,}
\end{cases}
\]
for all $\bx,\by \in \R^{\N}$. 

\begin{lemma}\label{l11}
For $\bx,\by \in \R^\N$ we have
\[
\sum_{j \in \N} | \ell_j(x_j,y_j) -1 | < \infty
\quad \Leftrightarrow \quad 
\bx - \by \in \X.
\]
In particular,
\[
\sum_{j \in \N} | v_j(x_j) -1 | < \infty
\quad \Leftrightarrow \quad 
\bx \in \X.
\]
\end{lemma}

\begin{proof}
Since $\prod_{j \in \N} \ell_j(x_j,y_j)$ converges and
$\ell_j(x_j,y_j) \neq 0$ for every $j \in \N$, the convergence of
$\sum_{j \in \N} | \ell_j(x_j,y_j) -1 |$ is equivalent to
$L(\bx,\by) \neq 0$, i.e., to $\bx-\by \in \X$.
The second statement is just a special case of the first
statement, namely with $\by := 0$. 
\end{proof}

Lemma \ref{l11} and Theorem \ref{ta1} immediately imply that $\Phi$
induces an isometric isomorphism between the incomplete tensor
product $G^{(\bv)}$ and the reproducing kernel Hilbert space
$H(L_\X)$. In this sense $H(L_\X)$
is the
incomplete tensor product of the Gaussian
spaces $H(\ell_j)$ based on the unit vectors $v_j = e_{0,j}$ for 
$j \in \N$. 

\begin{defi}
Let  $\bs := (\sigma_j)_{j \in \N}$ be a a sequence of shape parameters $\sigma_j > 0$.
The reproducing kernel  $L_\X$ defined by \eqref{def:kernel_L} on the domain 
$\X$ given by \eqref{eq:dom_gauss}
is called a \emph{Gaussian kernel} and the Hilbert space $H(L_\X)$ is called a \emph{Gaussian space} of functions of  infinitely many real variables.
\end{defi}

Larger domains $\Y \supsetneq \X$ in the definition
of $\Phi$ in \eqref{def:Phi} are discussed in the sequel.
The definition of $\tilde{L}$ in the following lemma is the 
counterpart to the representation \eqref{g55} of $\ell_j$ in the 
univariate case. Moreover, the lemma ensures that $\Phi$ is well-defined
for any choice of $\Y$.

\begin{lemma}\label{l48}
Let $\bx,\by \in \R^{\N}$. 
The series
\[
\tilde{L}(\bx,\by) := 
\sum_{\bn\in\bN} 
\prod_{j=1}^\infty e_{\nu_j,j}(x_j) \cdot e_{\nu_j,j}(y_j)
\]
is absolutely convergent, and 
\[
\tilde{L}(\bx,\by) =
\begin{cases}
L(\bx,\by), & \text{if $\bx,\by \in \X$,}\\
0, & \text{otherwise.}
\end{cases}
\]
\end{lemma}

\begin{proof}
Let $\bx \in \R^\N$ and $n \in \N$, and let $\bN_n$ denote the set of all 
sequences $\bn \in \bN$ with $\nu_j = 0$ for every $j > n$. Since 
\eqref{g55} yields
$\sum_{\nu \in \N_0} |e_{\nu,j}(x)|^2 = 1$ for $j \in \N$ 
and $x \in \R$, we obtain
\[
\sum_{\bn\in \bN_n} \prod_{j=1}^n |e_{\nu_j,j}(x_j)|^2 = 1.
\]
It follows that
\[
\sum_{\bn\in\bN_n} \prod_{j=1}^\infty |e_{\nu_j,j}(x_j)|^2 
= 
\prod^\infty_{j=n+1} |e_{0,j} (x_{j})|^2 = 
\exp \biggl( - 2 \sum^\infty_{j=n+1} \sigma_{j}^2 x_{j}^2 \biggr),
\]
and hereby
\[
\sum_{\bn\in\bN} \prod_{j=1}^\infty |e_{\nu_j,j}(x_j)|^2 
= \lim_{n\to\infty} 
\sum_{\bn\in\bN_n} \prod_{j=1}^\infty |e_{\nu_j,j}(x_j)|^2
=
\begin{cases}
1, & \text{if $\bx \in \X$,}\\
0, & \text{otherwise.}
\end{cases}
\] 
This yields the absolute convergence, as claimed, and
$\tilde{L}(\bx,\by)=0$ if $\bx \not\in \X$ or $\by \not\in \X$.

Next we verify $\tilde{L}(\bx,\by) = L(\bx,\by)$ for
$\bx,\by \in \X$. In the latter case
\[
\sum_{\bn\in\bN_n} \prod_{j=1}^\infty 
e_{\nu_j,j}(x_j) \cdot e_{\nu_j,j}(y_j)
= \prod_{j=1}^n \ell_j (x_j,y_j) \cdot
\exp\biggl(-\sum^\infty_{j=n+1}\sigma_{j}^2 \cdot (x_{j}^2+y_j^2)\biggr).
\]
In the limit $n \to \infty$ we obtain the claim.
\end{proof}

Because of its series representation, the function
$\tilde{L}$ is a reproducing kernel on $\R^\N \times \R^\N$. 
Since $L - \tilde{L} = L \cdot 1_{\Z \times \Z}$
with $\Z := \R^\N \setminus \X$, the function $L - \tilde{L}$ is a 
reproducing kernel on $\R^\N \times \R^\N$, too.

\begin{prop}\label{l53}
Assume that
\[
\X \subsetneq \Y \subseteq \R^\N.
\]
Then 
\begin{itemize}
\item[(i)]
$H(L_\Y)$ is the orthogonal sum of the closed, proper subspaces
$H(\tilde{L}_\Y)$ and $H(L_\Y-\tilde{L}_\Y)$,
\item[(ii)]
$H(\tilde{L}_\Y) = \{ f \in H(L_\Y) \colon f_{\Y \setminus \X} =0\}$,
\item[(iii)]
$\Phi$ is an isometric isomorphism between $G^{(\bv)}$ and 
$H(\tilde{L}_\Y)$. 
\end{itemize}
\end{prop}

\begin{proof}
For $\bx \in \Y \setminus \X$ we have $\tilde{L} (\bx,\bx) = 0$
and $(L-\tilde{L})(\bx,\bx) = 1$,
while $\tilde{L}(\bx,\bx)=1$ and $(L-\tilde{L})(\bx,\bx) = 0$ for 
$\bx \in \X$, see Lemma~\ref{l48}.
It follows that $H(\tilde{L}_\Y)$ and $H(L_\Y - \tilde{L}_\Y)$ are
non-trivial, but have a trivial intersection. This yields (i) and
(ii).

Let $\bx \in \Y \setminus \X$. Note that $\tilde{L}(\bx,\bx) = 0$ is 
equivalent to
\[
\prod_{j=1}^\infty e_{\nu_j,j}(x_j) = 0
\]
for every $\bn \in \bN$.
It follows that $\Phi f(\bx) = 0$ for every $f \in G^{(\bv)}$.
We conclude that $G^{(\bv)}$ is identified with 
$H(\tilde{L}_\Y)$ via $\Phi$, as claimed in (iii). 
\end{proof}

\begin{rem}\label{rem:summarize}
Lemma~\ref{l48} and Proposition~\ref{l53}
show that the proper choice of domain of the reproducing kernel 
$L$ defined in \eqref{def:kernel_L} is $\X := \ell^2(\bs^2)$, since this 
is the maximal domain such that the canonical mapping 
$\Phi \colon G^{(\bv)} \to \R^\X$ defined in \eqref{def:Phi} defines an 
isometric isomorphism between $G^{(\bv)}$ and 
$\Phi(G^{(\bv)}) = H(L_\X) = H(\tilde{L}_\X)$.
If we consider strictly larger domains $\Y$, then we still have 
$\Phi(G^{(\bv)}) = H(\tilde{L}_\Y)$, but $ H(\tilde{L}_\Y) \neq H(L_\Y)$. 
\end{rem}

\section{Isomorphisms between Gaussian Spaces and Hermite Spaces}\label{s:iso}

\subsection{Functions of a Single Variable}\label{s5.1}

It is known that the Gaussian kernel $\ell_\sigma$ with shape parameter 
$\sigma > 0$ can be represented in terms of the Hermite polynomials
$h_\nu$,
see, e.g., \citet[Sec.~4.3.1]{RW2006} and \citet[Sec.~3]{FHW2012}.
A starting point to derive such a representation is Mehler's formula
\[
\exp \left( \frac{1}{1-\beta^2} \cdot 
\left( \beta \cdot x y - \tfrac{1}{2} \beta^2 \cdot (x^2 + y^2) \right)
\right) = m(x,y)
\]
with
\[
m(x,y) := \left(1-\beta^2\right)^{1/2} \cdot 
\sum_{\nu \in \N_0} \beta^\nu \cdot h_\nu(x) \cdot h_\nu(y),
\]
which holds for all $x,y \in \R$ and $\beta \in {]0,1[}$,
see \citet[Eqn.~(22) on p. 194]{HTF1953}. Note that 
\begin{equation}\label{eq201}
\sum_{\nu \in \N_0} \beta^\nu \cdot |h_\nu(x)|^2 < \infty
\end{equation}
for every $x \in \R$ already follows from Cram\'er's inequality, 
see Lemma~\ref{l41}.
We add that Mehler's formula also yields a closed
form representation of Hermite kernels in the case (EG) with $b=1$,
see \citet[p.~186]{IL15}.

Using
\[
\beta \cdot x y - \tfrac{1}{2} \beta^2 \cdot (x^2 + y^2)
=
- \tfrac{1}{2} \beta \cdot (x-y)^2 + \tfrac{1}{2} \beta (1-\beta)
\cdot (x^2 + y^2)
\]
and introducing a scaling parameter $c>0$,
Mehler's formula may be rewritten as
\[
\exp \left( - \frac{c^2 \beta}
{2 (1-\beta^2)} \cdot (x-y)^2\right) 
= 
\exp \left( - \tau \cdot (x^2 + y^2) \right) 
\cdot m(c x,c y)
\]
with
\[
\tau := \frac{c^2 \beta} {2 (1+\beta)}.
\]
For $f \in L^2(\mu_0)$ and $x \in \R$ we define
\[
qf (x) := c^{1/2} \cdot \exp (-\tau \cdot x^2) \cdot f(c x).
\]
If
\begin{equation}\label{eq203}
\frac{c^2 \beta}{2 (1-\beta^2)} = \sigma^2
\end{equation}
then
\[
\ell_\sigma(x,y) =
\frac{\left(1-\beta^2\right)^{1/2}}{c} \cdot 
\sum_{\nu \in \N_0} \beta^\nu \cdot q h_\nu (x) \cdot q h_\nu(y)
\]
for all $x,y \in \R$.

\begin{lemma}\label{l200}
The mapping $q$ is an isometric isomorphism on $L^2(\mu_0)$ if and
only if
\begin{equation}\label{eq200}
\tau = \frac{c^2-1}{4}.
\end{equation}
\end{lemma}

\begin{proof}
Let $t \colon \R \to \R$ be given by $t(x) := c x$.
The image measure $t\mu_0$ of $\mu_0$ with respect to $t$ is
the normal distribution with zero mean and variance $c^2$.
The density of $t\mu_0$ with respect to $\mu_0$
is therefore given by the density ratio
\[
\phi (x) :=
c^{-1} \cdot 
\exp\left(- \frac{x^2}{2 c^2}\right) \cdot
\exp\left(\frac{x^2}{2}\right)
= 
c^{-1} \cdot 
\exp \left( \frac{c^2-1}{2 c^2} \cdot x^2 \right).
\]
For $f \in L^2(\mu_0)$ this yields
\begin{align*}
\int_\R |q f (x)|^2 \, d \mu_0(x)
&=
c \cdot \int_\R
\exp \left( - 2 \tau/c^2 \cdot |t(x)|^2\right) \cdot |f(t(x))|^2 \,
d \mu_0 (x) \\
&=
c \cdot \int_\R
\phi(x) \cdot \exp \left( - 2 \tau/c^2 \cdot |x|^2\right) \cdot |f(x)|^2 \,
d \mu_0 (x). 
\end{align*}
We conclude that $q$ is an isometry on $L^2(\mu_0)$ if and only if 
\[
\frac{c^2-1}{2 c^2} = \frac{2\tau}{c^2},
\] 
which is equivalent to \eqref{eq200}. 
Assume that \eqref{eq200} is satisfied. It follows that
\[
qf (x) = \left|\phi(t(x))\right|^{-1/2} \cdot f(t(x)).
\]
For $g \in L^2(\mu_0)$ and $f(x) := \left|\phi(x)\right|^{1/2}
\cdot g\left(t^{-1}(x)\right)$ we obtain $f \in L^2(\mu_0)$ as well
as $qf = g$.
\end{proof}

\begin{lemma}\label{l201}
For every $\sigma > 0$ the unique solution 
$(\beta,c) \in {]0,1[} \times {]0,\infty[}$ of
\eqref{eq203} and \eqref{eq200} is given by
\begin{equation}\label{eq202}
c = \left(1+8\sigma^2\right)^{1/4}
\qquad \text{and} \qquad
\beta 
= 1 - \frac{2}{1+c^2}.
\end{equation}
\end{lemma}

\begin{proof}
Let $(\beta,c) \in {]0,1[} \times {]0,\infty[}$ satisfy
\eqref{eq203} and \eqref{eq200}.
From \eqref{eq200} we obtain
\[
c^2 = \frac{1+\beta}{1-\beta},
\]
and together with \eqref{eq203} this yields
\[
\frac{\beta}{2 (1-\beta)^2} = \sigma^2.
\]
Since $\beta \mapsto \beta/(1-\beta)^2$ is strictly increasing
on ${]0,1[}$, we conclude that the first component $\beta$ of
the solution of \eqref{eq203} and \eqref{eq200} is uniquely determined.
Due to \eqref{eq203} the same holds true for the second component $c$.

Conversely, it is straightforward to verify that $c$ and $\beta$ 
according to \eqref{eq202} belong to ${]0,\infty[}$ and ${]0,1[}$,
respectively, and satisfy \eqref{eq203} and \eqref{eq200}.
\end{proof}

In the sequel we assume that $c$ and $\beta$ are chosen according
to \eqref{eq202}, and we introduce the Fourier weights
\[
\phantom{\qquad\quad \nu \in \N_0.}
\alpha_{\nu} 
:= \frac{1}{(1-\beta)\cdot \beta^\nu},
\qquad\quad \nu \in \N_0.
\]
Since
\[
\frac{\left(1-\beta^2\right)^{1/2}}{c} = 1 - \beta,
\]
we obtain the representation
\begin{equation}\label{rep_ell_sigma}
\phantom{\qquad\quad x,y \in \R,}
\ell_\sigma(x,y) =
\sum_{\nu \in \N_0} \alpha_\nu^{-1} \cdot q h_\nu (x) \cdot q h_\nu(y),
\qquad\quad x,y \in \R,
\end{equation}
of the Gaussian kernel $\ell_\sigma$ from Mehler's formula and
Lemma~\ref{l201}. Cf.\ \citet[p.~257]{FHW2012},
who consider the normal distribution with zero mean and
variance $1/2$, instead of the standard normal distribution $\mu_0$,
together with unnormalized Hermite polynomials.

Due to Lemma~\ref{l200} and Lemma~\ref{l201} the functions $qh_\nu$ with 
$\nu \in \N_0$ form an orthonormal basis of $L^2(\mu_0)$, and 
\eqref{eq201} corresponds to \eqref{g23}. We conclude that 
Proposition~\ref{l24},
with the orthonormal basis $(q h_\nu)_{\nu \in \N_0}$, provides a 
characterization of $H(\ell_\sigma)$ as a linear 
subspace of $L^2(\mu_0)$. 

We relate the Gaussian space $H(\ell_\sigma)$ to the Hermite space
$H(k_\sigma)$ with 
\begin{equation}\label{def_k_sigma}
\phantom{\qquad\quad x,y \in \R.}
k_\sigma(x,y) := 
\sum_{\nu \in \N_0}
\alpha_{\nu}^{-1} \cdot h_\nu(x) \cdot h_\nu(y),
\qquad\quad x,y \in \R.
\end{equation}

\begin{prop}\label{l10}
The mapping $q|_{H(k_\sigma)}$ is 
an isometric isomorphism between $H(k_\sigma)$ and $H(\ell_\sigma)$.
\end{prop}

\begin{proof}
We use Lemma~\ref{l200} and Lemma~\ref{l201} as well as 
\eqref{rep_ell_sigma} and Proposition~\ref{l24}.
Let $c_\nu \in \R$, $\nu \in \N_0$, satisfy 
$\sum_{\nu \in \N_0} \alpha_{\nu} \cdot |c_\nu|^2 < \infty$.
Then $f := \sum_{\nu \in \N_0} c_\nu \cdot h_\nu$ is in $H(k_\sigma)$, 
and the sum converges in $L^2(\mu_0)$ and absolutely at every 
point $x \in \R$. Observe that
the pointwise convergence is preserved by $q$, and therefore
$q f (x) = \sum_{\nu \in \N_0} c_\nu \cdot q h_\nu(x)$.
Moreover, $qf = \sum_{\nu \in \N_0} c_\nu \cdot q h_\nu$ with
convergence in $L^2(\mu_0)$. It follows that $q f \in H(\ell_\sigma)$
and $\|f\|_{H(k_\sigma)} = \|q f\|_{H(\ell_\sigma)}$.
Identity~\eqref{eq:charRKHS}, applied to $H(k_\sigma)$ as well as to 
$H(\ell_\sigma)$, yields that $q(H(k_\sigma)) = H(\ell_\sigma)$.
\end{proof}

\begin{rem}\label{r13}
We combine Proposition~\ref{l24}, 
Lemma~\ref{l200}, Lemma~\ref{l201}, and 
Proposition~\ref{l10}
to conclude that
\begin{center}
\begin{tikzcd}
H(k_\sigma)\arrow[r,hook]\arrow[d,"q|_{H(k_\sigma)}"]
&
L^2(\mu_0)
\arrow[d,"q"]
\\
H(\ell_\sigma)\arrow[r,hook]
&
L^2(\mu_0)
\end{tikzcd}
\end{center}
is a commutative diagram of bounded linear operators with isometric
isomorphisms in the vertical direction. 
\end{rem}

We close this subsection by commenting on inclusion relations 
between Gaussian spaces and Hermite spaces.

\begin{rem}\label{Rem_5.4}
Consider an arbitrary Hermite kernel $k$ and an arbitrary Gaussian
kernel $\ell_\sigma$. Since 
\[
|f(x) - f(y)|^2 \leq 
2 \, \|f\|^2_{H(\ell_\sigma)} \cdot (1 - \ell_\sigma(x,y))
\]
for all $f \in H(\ell_\sigma)$ and $x,y \in \R$, the functions from
$H(\ell_\sigma)$ are bounded. Therefore the space $H(\ell_\sigma)$
does not contain any non-trivial polynomial, while $H(k)$ obviously
contains the whole space $\Pi$ of polynomials. 
In particular,
\[
H(k) \not\subseteq  H(\ell_\sigma).
\]

Observe that $\ell_\sigma$ is translation-invariant.
General conditions for translation-invariant kernels 
$\ell$ that guarantee $H(\ell) \cap \Pi = \{0\}$ 
have recently been established in \citet{DZ21} and \citet{Karvonen21}.
\end{rem}

\begin{rem}\label{r99}
We briefly discuss inclusions of Gaussian spaces
in Hermite spaces. More precisely, we consider Hermite spaces
$H(k)$ with Fourier weights $\alpha_\nu := \beta^{-\nu}$ for
any $\beta \in {]0,1[}$.
The following holds true for every $\sigma > 0$.
There exist $0 < \beta^{(1)} \leq \beta^{(2)} < 1$ such that
\[
H(\ell_\sigma) \not\subseteq H(k) \qquad 
\text{if $0 < \beta < \beta^{(1)}$}
\]
and
\[
H(\ell_\sigma) \subsetneq H(k) \qquad 
\text{if $\beta^{(2)} < \beta < 1$.}
\]
In particular, for $\beta$ according to~\eqref{eq202}, where
$H(k) = H(k_\sigma)$ as vector spaces and
$q|_{H(k_\sigma)}$ is an isometric isomorphism between $H(k_\sigma)$ and 
$H(\ell_\sigma)$,
we have $\beta < \beta^{(1)}$, and therefore 
\[
H(\ell_\sigma) \not\subseteq H(k_\sigma). 
\]
So far, our analysis only yields values $\beta^{(1)} <
\beta^{(2)}$; therefore we skip the proof.
\end{rem}

\subsection{Functions of Infinitely Many Variables}

As in Section~\ref{s4.2},
we consider again a sequence $\bs$ of shape parameters
$\sigma_j > 0$ and the corresponding sequence of Gaussian
kernels $\ell_j$. We study Gaussian spaces and Hermite spaces
of functions of infinitely many variables with Fourier weights
given by 
\[
\phantom{\qquad\quad \nu \in \N_0,\ j \in \N,}
\alpha_{\nu, j} := \frac{1}{(1-\beta_j)\cdot \beta^\nu_j},
\qquad\quad \nu \in \N_0,\ j \in \N,
\]
where 
\begin{equation*}
c_j := \left(1+8\sigma_j^2\right)^{1/4}
\qquad \text{and} \qquad
\beta_j 
:= 1 - \frac{2}{1+c_j^2}\,,
\end{equation*}
cf. Lemma~\ref{l201},
under the assumption
\begin{equation}\label{sigma_squared}
\sum_{j\in\N} \sigma^2_j < \infty.
\end{equation}
Note that \eqref{sigma_squared} is equivalent to 
\begin{equation}\label{g69}
\sum_{j \in \N} |c_j - 1| < \infty.
\end{equation}
As in Section~\ref{s4.2}, we consider the domain $\X := \ell^2(\bs^2)$.

We will establish the counterpart to Remark~\ref{r13} 
via tensorization and identification of the corresponding
incomplete tensor product spaces with the reproducing kernel Hilbert
spaces $H(K)$ and $H(L_\X)$ or the space $L^2(\mu)$,
respectively. 

\begin{lemma}\label{lemma:5.7}
Assume that \eqref{sigma_squared} is satisfied. Then we have
{\rm (\ha)}--{\rm (\hc)}, and $\X$ is the maximal domain of the kernel $K$ 
and satisfies $\mu(\X) =1$.
\end{lemma}

\begin{proof}
From \eqref{sigma_squared} we obtain
the weak asymptotics
\[
\alpha^{-1}_{1,j} = (1-\beta_j) \cdot \beta_j 
\asymp \beta_j \asymp \sigma^2_j,
\]
the convergence $\sum_{j \in \N} \beta_j < \infty$,
and $\ell^2(\ba^{-1}_1) = \ell^2(\bs^2)$ for
$\ba^{-1}_1 := (\alpha_{1,j}^{-1})_{j \in \N}$.

Let $r_j > 0$ be defined by $\beta_j = 2^{-r_j}$. 
Note that the Fourier weights 
\[
\phantom{\qquad\quad \nu \in \N_0,\ j \in \N,}
\beta_j^{-\nu} = 2^{r_j \cdot \nu},
\qquad\quad \nu \in \N_0,\ j \in \N,
\]
are of the form (EG) with $b_j := 1$.
Due to Lemma~\ref{l2} we have (\ha)--(\hc) in this case, and 
$\ell(\bs^2)$ is the maximal domain of the corresponding Hermite kernel,
see Proposition~\ref{Lemma:Char_X_EG}.

It remains to consider
$\alpha_{\nu,j} = (1-\beta_j)^{-1} \cdot \beta_j^{-\nu}$ instead
of $\beta_j^{-\nu}$. Since
\[
0 < \inf_{j \in \N} (1-\beta_j) \leq \sup_{j \in \N} 
(1-\beta_j) = 1,
\] 
(\ha), (\hc), and the maximal domain are not affected by this
change. Since 
\[
\alpha_{0,j} - 1 = \frac{\beta_j}{1-\beta_j},
\]
we also have (\hb). Due to Proposition~\ref{lem:domain2}
we have $\mu(\X) =1$.
\end{proof}

Put 
\[
c_\ast := \prod_{j=1}^\infty c_j > 1,
\]
cf.\ \eqref{g69}. Let $t \colon \R^\N \to \R^\N$ and 
$\phi \colon \X \to \left[0,\infty\right[$ be given by
\[
t(\bx) := (c_1 \, x_1, c_2 \, x_2, \dots)
\]
for $\bx \in \R^\N$ and by
\[
\phi(\bx) := c^{-1}_\ast \cdot
\exp \biggl( \sum_{j \in \N} \frac{c_j^2-1}{2c_j^2} \, x_j^2 \biggl)
\]
for $\bx \in \X$.
For $f \in L^2(\mu)$ and $\bx \in \X$ we define 
\[
Qf (\bx) := 
\Bigl( \phi \bigl(t(\bx)\bigr) \Bigr)^{-1/2} \cdot f \bigl(t(\bx)\bigr)
= c^{1/2}_\ast \cdot 
\exp \biggl( - \sum_{j\in\N} \frac{c_j^2-1}{4} \,x_j^2 \biggr) 
\cdot f\bigl(t(\bx)\bigr). 
\]

\begin{theo}\label{t1}
Assume that \eqref{sigma_squared} is satisfied. Then
the mapping $Q$ is an isometric isomorphism on $L^2(\mu)$, and its 
restriction $Q|_{H(K)}$ is an 
isometric isomorphism
between $H(K)$ and $H(L_\X)$.
\end{theo}

\begin{proof}
At first we aim at the space $L^2(\mu)$ and the mapping $Q$. Let 
\[
L^{(\be)} := \bigotimes_{j \in \N} (L^2(\mu_0))^{(1)}
\qquad \text{and} \qquad
L^{(\bw)} := \bigotimes_{j \in \N} (L^2(\mu_0))^{(w_j)}
\]
denote the incomplete tensor products of the spaces $L^2(\mu_0)$
based on the constant function $1$ as the unit vector
for every $j \in \N$ and based on the unit vectors
\[
w_j := q_j 1
\]
for $j \in \N$, respectively.
We use Lemma~\ref{l200} and Lemma~\ref{l201}
to conclude that the tensor product
\[
q := \bigotimes_{j\in\N} q_j \colon L^{(\be)} \to L^{(\bw)}
\]
of the operators $q_j \colon L^2(\mu_0) \to L^2(\mu_0)$ is well-defined 
and an isometric isomorphism.

By definition $w_j (x) = c_j^{1/2} \cdot \exp(-(c_j^2-1)/ 4 \cdot x^2)$.  
We use 
$0 < w_j \leq c_j^{1/2}$ and $\|w_j\|_{L^2(\mu_0)}=1$ to obtain
$\scp{w_j}{1}_{L^2(\mu_0)} = \|w_j\|_{L^1(\mu_0)}$ and
\[
1 \geq \|w_j\|_{L^1(\mu_0)}  
\geq
c_j^{-1/2} \cdot \int_{\R} w_j^2 \, d \mu_0 = c_j^{-1/2}.
\] 
Therefore
\begin{equation}\label{g67}
\sum_{j \in \N} |\scp{w_j}{1}_{L^2(\mu_0)} - 1| 
\leq
\sum_{j \in \N} (1 - c_j^{-1/2}) < \infty,
\end{equation}
cf.\ \eqref{g69}.
The estimate \eqref{g67} together with \eqref{g62} implies
\[
L^{(\bw)} = L^{(\be)}.
\]

Let $\Psi \colon L^2(\mu) \to L^{(\be)}$ denote the isometric isomorphism 
according to Theorem~\ref{ta2} and
\[
\widetilde{Q} := 
\Psi^{-1} \circ q \circ \Psi \colon L^2(\mu) \to L^2(\mu).
\]
We claim that $\widetilde{Q} = Q$.
Recall
that the functions $h_\bn$ with $\bn \in \bN$ form an
orthonormal basis of $L^2(\mu)$. 
By definition of $q$ and $h_\bn$,
\[
(q \circ \Psi) h_\bn = q \bigotimes_{j \in \N} h_{\nu_j} =
\bigotimes_{j \in \N} q_j h_{\nu_j}.
\]
We combine \eqref{g67} and Theorem~\ref{ta2} to obtain
\[
\widetilde{Q} h_\bn (\bx) = 
\prod_{j=1}^\infty q_j h_{\nu_j} (x_j) = Q h_\bn (\bx)
\]
for $\rho$-a.e.\ $\bx \in \X$.
For $f \in L^2(\mu)$ we have
\[
\widetilde{Q} f = 
\sum_{\bn \in \bN} \scp{f}{h_\bn}_{L^2(\mu)} \cdot
\widetilde{Q} h_\bn
=
\sum_{\bn \in \bN} \scp{f}{h_\bn}_{L^2(\mu)} \cdot Q h_\bn
\]
with convergence in $L^2(\mu)$. Observe that $Q$ preserves pointwise 
convergence. Redefining the functions $h_\bn$ on sets of measure zero we
also have 
\[
Q f = \sum_{\bn \in \bN} \scp{f}{h_\bn}_{L^2(\mu)} \cdot Q h_\bn
\]
with pointwise convergence along a subsequence. It follows that
\[
Q f = \widetilde{Q} f \in L^2(\mu),
\]
as claimed, and therefore $Q$ is an isometric isomorphism on $L^2(\mu)$.

Proceeding analogously to the first part of the proof we now
aim at the spaces $H(K)$ and $H(L_\X)$ and the mapping
$Q|_{H(K)}$. Let 
\[
H^{(\ba_0^{-1/2})} := \bigotimes_{j \in \N} (H(k_j))^{(\alpha_{0,j}^{-1/2})}
\]
denote the incomplete tensor products of the spaces $H(k_j)$
based on the constant functions $\alpha_{0,j}^{-1/2}$ for $j \in \N$ as unit vectors.
Furthermore, let
\[
G^{(\bv)} := \bigotimes_{j \in \N} (H(\ell_j))^{(v_j)}
\qquad \text{and} \qquad
G^{(\widehat{\bw})} := \bigotimes_{j \in \N} (H(\ell_j))^{(\widehat{w}_j)}
\]
denote the incomplete tensor products of the spaces
$H(\ell_j)$ based on the vectors $v_j := \ell_j(\cdot,0)$
and
\[
\widehat{w}_j := q_j \alpha_{0,j}^{-1/2} = \alpha_{0,j}^{-1/2} w_j
\]
for $j \in \N$, respectively.
We use Lemma~\ref{l10} to conclude that the tensor product
\[
\widehat{q} := \bigotimes_{j\in\N} q_j|_{H(k_j)} \colon H^{(\ba_0^{-1/2})} \to G^{(\widehat{\bw})}
\]
of the operators $q_j|_{H(k_j)} \colon H(k_j) \to H(\ell_j)$ is
well-defined and an isometric isomorphism.

Since
\[
\scp{\widehat{w}_j}{v_j}_{H(\ell_j)} 
= \widehat{w}_j(0) = \alpha_{0,j}^{-1/2} \cdot (q_j 1 )(0) =
(c_j/\alpha_{0,j})^{1/2},
\]
we obtain
\begin{equation}\label{g68}
\sum_{j \in \N} |\scp{\widehat{w}_j}{v_j}_{H(\ell_j)} - 1|
\leq
\sum_{j \in \N} |(c_j^{1/2}-1) / \alpha_{0,j}^{1/2}| +
\sum_{j \in \N} |\alpha_{0,j}^{-1/2} - 1| < \infty,
\end{equation}
cf.\ \eqref{g69} and (\hb).
The estimate \eqref{g68} together with \eqref{g62} implies
\[
G^{(\widehat{\bw})} = G^{(\bv)}.
\]

Let $\Phi_1 \colon H^{(\ba_0^{-1/2})} \to H(K)$ and
$\Phi_2 \colon G^{(\bv)} \to H(L_\X)$ 
denote the isometric isomorphisms
according to Theorem~\ref{ta1} and 
\[
\widehat{Q}
:= \Phi_2 \circ \widehat{q} \circ \Phi_1^{-1} \colon H(K) \to H(L_\X).
\]
We claim that $\widehat{Q} = Q|_{H(K)}$.
Recall 
that the functions $\alpha_\bn^{-1/2} h_\bn$ with $\bn \in \bN$ 
form an orthonormal basis of $H(K)$. 
By definition of $\widehat{q}$ and $h_\bn$,
\[
(\widehat{q} \circ \Phi_1^{-1}) h_\bn = \widehat{q} \bigotimes_{j \in \N} h_{\nu_j} =
\bigotimes_{j \in \N} q_j h_{\nu_j}.
\]
We combine \eqref{g68} and Theorem~\ref{ta1} to obtain
\[
\widehat{Q} h_\bn (\bx) = 
\prod_{j=1}^\infty q_j h_{\nu_j} (x_j) = Q h_\bn (\bx)
\]
for every $\bx \in \X$.
For $f \in H(K)$ and $\bx \in \X$ this implies 
\begin{align*}
\widehat{Q} f (\bx) &= 
\sum_{\bn \in \bN} \alpha_\bn^{-1} \cdot 
\scp{f}{h_\bn}_{H(K)} \cdot \widehat{Q} h_\bn (\bx)
=
\sum_{\bn \in \bN} \alpha_\bn^{-1} \cdot 
\scp{f}{h_\bn}_{H(K)} \cdot Q h_\bn (\bx) \\
&= Qf (\bx),
\end{align*}
which completes the proof.
\end{proof}

\begin{rem}\label{r23}
According to Proposition~\ref{l24} and Theorem~\ref{t1},
\begin{center}
\begin{tikzcd}
H(K)\arrow[r,hook]\arrow[d,"Q|_{H(K)}"]
&
L^2(\mu) 
\arrow[d,"Q"]
\\
H(L_\X)\arrow[r,hook]
&
L^2(\mu)
\end{tikzcd}
\end{center}
is a commutative diagram of bounded linear operators with isometric
isomorphisms in the vertical direction.

Another important feature of $Q$ 
is the obvious fact that for every $\bx \in \X$ a single function value 
of $f \in H(K)$ suffices to compute $Qf(\bx)$ and vice versa.

These results allow to transfer computational problems from the  
space $H(L_\X)$ onto the space $H(K)$, as it will be shown in a
forthcoming paper, cf.\ Remark~\ref{r21}. In this approach it is
preferable to replace the Fourier weight $\alpha_{0,j}$ by one in the 
definition of each of the kernels $k_j$; the corresponding tensor product 
kernel $K^\prime$ has already been considered in Section~\ref{s3.2}.
We obtain $H(K) = H(K^\prime)$ as vector spaces and, as a minor change, 
that $Q|_{H(K^\prime)}$ is an isomorphism between $H(K^\prime)$
and $H(L_\X)$ with norm $\prod_{j=1}^\infty (1-\beta_j)^{-1/2}$;
its inverse has norm one. See Lemma~\ref{l43}.
\end{rem}

\begin{rem}
There is
a one-to-one correspondence between the scale of Hermite
spaces $H(K)$ appearing in Theorem~\ref{t1} and the scale of Hermite 
spaces with Fourier weights of exponential growth (EG) with $b_j = 1$ 
for all $j \in \N$. Indeed, the change of weights 
 $(1-\beta_j)^{-1} \beta_j^{-\nu} \mapsto \beta_j^{-\nu}$ for
$\nu\in\N_0$ and $j\in \N$
leaves the vector spaces $H(K)$ invariant and leads to different, but 
equivalent norms, cf.\ the proof of Lemma~\ref{lemma:5.7}.
It is easy to see that the mapping between the scales 
of Hermite spaces induced by this change of weights is well-defined and 
bijective.
\end{rem}

\appendix

\section{Countable Tensor Products of Hilbert Spaces}\label{a1}

Tensor products of arbitrary families of Hilbert spaces have been
introduced (actually named complete direct products) and thoroughly 
studied in \citet{Neu39}. In the present paper we are interested in 
countably infinite tensor products, and thus we consider a sequence 
$(H_j)_{j \in \N}$ of Hilbert spaces $H_j \neq \{0\}$ over the same 
scalar field $\K \in \{\R,\C\}$. 

As shown in Sections~\ref{as2} and \ref{as3}, the main results from 
\citet{Neu39} on complete and incomplete tensor products of Hilbert 
spaces can be represented within the framework of reproducing kernel 
Hilbert spaces in a concise way.
Tensor products of bounded linear operators are discussed in
Section~\ref{as6}.

In Sections~\ref{as4} and \ref{as5} we consider particular cases,
where the results are due to \citet{Rue20} and \citet{Gui69}, 
respectively.
In Sections~\ref{as4}
we study tensor products of reproducing kernel Hilbert spaces
and their relation to Hilbert spaces with
reproducing kernels of tensor product form. 
In Sections~\ref{as5}
we consider tensor products of $L^2$-spaces and their
relation to $L^2$-spaces with respect to the corresponding
product measures.

\subsection{Convergence of Products}\label{as1}

We discuss the convergence of infinite products in $\K$,
see \citet[Chap.~2]{Neu39} and cf.\ \eqref{g29}.
Let $(z_j)_{j \in \N}$ denote a sequence in $\K$.
We say that the product $\prod_{j \in \N} z_j$ is 
\emph{convergent} with value 
$z \in \K$, if the following holds for every $\eps >0$. 
There exists a finite set $I_0 \subseteq \N$ such that
$|\prod_{j \in I} z_j - z| \leq \eps$ holds
for every finite set $I_0 \subseteq I \subseteq \N$. The value
of a convergent product is uniquely determined; hence we 
put $\prod_{j \in \N} z_j := z$, given convergence with value $z$.
In the case of convergence we obviously also have convergence of
$\bigl(\prod_{j=1}^J z_j\bigr)_{J \in \N}$, as required in
\eqref{g29}, and 
$\prod_{j \in \N} z_j = \prod_{j=1}^\infty z_j$.

Moreover, we say that $\prod_{j \in \N} z_j$ is 
\emph{quasi-convergent}, if
$\prod_{j \in \N} |z_j|$ converges. The latter is necessary, but
not sufficient for convergence of $\prod_{j \in \N} z_j$,
and for non-convergent, but quasi-convergent products we
put $\prod_{j \in \N} z_j := 0$. The following properties
are equivalent:
\begin{itemize}
\item[(i)] 
$\prod_{j \in \N} z_j$ is convergent and $\prod_{j \in \N} z_j \neq 0$,
\item[(ii)]
$\prod_{j \in \N} z_j$ is quasi-convergent and 
$\prod_{j \in \N} z_j \neq 0$,
\item[(iii)]
$\sum_{j \in \N} |z_j-1| < \infty$ and $z_j \neq 0$ for every $j
\in \N$.
\end{itemize}
Furthermore, $\sum_{j \in \N} |z_j-1| < \infty$ implies
$\lim_{J \to \infty} \prod_{j=J}^\infty z_j = 1$.
If we have $z_j \in \R$ with $z_j \geq 1$ for
every $j \in \N$, then the convergence of $\prod_{j \in \N} z_j$
is obviously equivalent to the convergence of 
$\bigl(\prod_{j=1}^J z_j\bigr)_{J \in \N}$.

\subsection{The Complete Tensor Product}\label{as2}

In the sequel we use the notation $\bof = (f_j)_{j \in \N}$
with $f_j \in H_j$ for elements $\bof \in \times_{j \in \N} H_j$. 
For the construction of the complete tensor product of
the spaces $H_j$ we define
\[
C := \Bigl\{ \bof \in \times_{j \in \N} H_j \colon
\text{$\prod_{j \in \N} \|f_j\|_{H_j}$ converges}\Bigr\}.
\]
For $\bof, \bog \in C$ the product
$\prod_{j \in \N} \scp{f_j}{g_j}_{H_j}$ is quasi-convergent,
see \citet[Lem.~ 2.5.2]{Neu39}. Hence we may define a mapping
$\KK \colon C \times C \to \K$ by
\[
\phantom{\qquad\quad \bof,\bog \in C,}
\KK(\bog,\bof) :=
\prod_{j \in \N} \scp{f_j}{g_j}_{H_j},
\qquad\quad \bof,\bog \in C,
\]
and \citet[Lem.~3.4.1]{Neu39} implies that $\KK$ is a reproducing kernel.

Cf.\ \citet[Def.~3.5.1]{Neu39} for the following definition.

\begin{defi}
The \emph{complete tensor product} of the spaces $H_j$ is the 
Hilbert space 
\[
H := \bigotimes_{j \in \N} H_j := H(\KK)
\]
with reproducing kernel $\KK$.
\end{defi}

For
any $\bof \in C$ 
the function 
\[
\bigotimes_{j \in \N} f_j := \KK(\cdot,\bof) \in H
\]
is called an \emph{elementary tensor}. Its norm is 
$\| \bigotimes_{j \in \N} f_j \|_H = \prod_{j \in \N} \|f_j\|_{H_j}$, and
the span of elementary tensors is dense in $H$.
 
It is desirable to work on smaller domains $\widetilde{C} \subsetneq C$
consisting of sequences that are easier to handle. If the restriction 
mapping $f\mapsto f|_{\widetilde{C}}$ defined on $H(\KK)$ is injective, 
it is actually already an isometric isomorphism onto 
$H(\KK|_{\widetilde{C} \times \widetilde{C}})$.
Hence we may identify in a natural way the reproducing kernel Hilbert 
spaces $H=H(\KK)$ and $H(\KK|_{\widetilde{C} \times \widetilde{C}})$. 

In this sense, it suffices to consider the elements of $H$
on the domain
\[
C_0 := \Bigl\{ \bof \in \times_{j \in \N} H_j \colon
\sum_{j \in \N} \left| \|f_j\|_{H_j}-1 \right| < \infty \Bigr\}
\subsetneq C.
\]
Indeed, for $\bof \in C$ we have 
$\KK(\bof,\bof) \neq 0$ if and only if $\bof \in C_0$ and
$f_j \neq 0$ for every $j \in \N$. In particular, 
$\KK(\cdot,\bof) = 0$ for $\bof \not \in C_0$, so that 
all elements $g \in H$ vanish on $C \setminus C_0$.

Consider $\bof \in C$ and a sequence $(z_j)_{j \in \N}$ in $\K$ with a
convergent product. For $z := \prod_{j \in \N} z_j$ and 
$\bof^\prime$ given by $f^\prime_j := z_j \cdot f_j$ we have
$\bof^\prime \in C$ and
\[
z \cdot \KK(\cdot,\bof) = \KK(\cdot,\bof^\prime),
\]
see \citet[Lem.~3.3.6]{Neu39}.
This property allows to further reduce the domain, namely  
to consider the elements of $H$ only on
\[
V := \{ \bof \in \times_{j \in \N} H_j \colon 
\text{$\|f_j\|_{H_j} = 1$ for every $j \in \N$}\}
\subsetneq C_0.
\]
Indeed, for $\bof \in C$ with $\KK(\cdot,\bof) \neq 0$ and 
$\bof^\prime \in V$ given by
\begin{equation}\label{g40}
f^\prime_j := \|f_j\|^{-1}_{H_j} \cdot f_j
\end{equation} 
we obtain
\[
\KK(\cdot,\bof) = 
\prod_{j \in \N} \|f_j\|_{H_j} \cdot \KK(\cdot,\bof^\prime).
\]

\subsection{The Incomplete Tensor Product}\label{as3}

Incomplete tensor products are subspaces of $H$ and constructed in the 
following way. The property
\[
\sum_{j \in \N} \left| \scp{f_j}{g_j}_{H_j}-1 \right| < \infty
\]
for $\bof, \bog \in C_0$ defines an equivalence relation on $C_0$,
see \citet[Lem.~3.3.3]{Neu39}. Observe that $\bog \in C_0$ has 
at most a finite number of components $g_j = 0$. If we replace
all of them by non-zero elements of the corresponding
Hilbert spaces $H_j$, we obtain $\bof \in C_0$, equivalent to
$\bog$, with $\KK(\cdot,\bof) \neq 0$.
For $\bof \in C_0$ with $\KK(\cdot,\bof) \neq 0$
and $\bof^\prime \in V$ given by \eqref{g40}
we have the equivalence of $\bof$ and $\bof^\prime$.
For a sequence $\bv \in V$ of unit vectors $v_j \in H_j$ 
we consider its equivalence class 
\[
\Co^{(\bv)} := \Bigl\{ \bof \in C_0 \colon
\sum_{j \in \N} \left| \scp{v_j}{f_j}_{H_j}-1 \right| < \infty \Bigr\}.
\]
Observe that $\{\Co^{(\bv)} \colon \bv \in V\}$ is a partition 
of the set $C_0$.
Cf.\ \citet[Def.~4.1.1]{Neu39} for the following definition.

\begin{defi}\label{def:incprod}
The \emph{incomplete tensor product}
of the spaces $H_j$ based on $\bv \in V$
is the closed linear subspace
\[
H^{(\bv)} := 
\bigotimes_{j \in \N} H_j^{(v_j)} :=
\overline{\spann} \{\KK(\cdot,\bof) \colon \bof \in
\Co^{(\bv)}\}
\]
of $H$, equipped with the induced norm.
\end{defi} 

It follows that for every $\bw \in V$
\begin{equation}\label{g62}
\bw \in \Co^{(\bv)} \quad \Rightarrow \quad
H^{(\bv)} =H^{(\bw)}
\end{equation}
and
\[
\bw \not\in \Co^{(\bv)} \quad \Rightarrow \quad
H^{(\bv)} \perp H^{(\bw)}.
\]
The latter orthogonality property follows from
$\KK(\bog,\bof) = 0$ for all 
non-equivalent $\bof,\bog \in C_0$.
Consequently, it suffices to consider the elements of $H^{(\bv)}$ on
the domain $\Co^{(\bv)}$, since every $g \in H^{(\bv)}$
vanishes on $C \setminus \Co^{(\bv)}$.

The observations above yield the following representation of $H$ as an 
orthogonal sum of Hilbert spaces: If $\mathcal{R} \subseteq V$ is a 
system of representers of the equivalence classes 
$\Co^{(\bv)}$, $\bv \in V$, then 
\[
H = \bigoplus_{\bv \in \mathcal{R}} H^{(\bv)}.
\]

The following approximation is employed in the proof of
\citet[Lem.~4.1.2]{Neu39}. Because of its particular relevance
for this paper we present the result together with a proof.

\begin{lemma}\label{la6}
Let $\bof \in \Co^{(\bv)}$. For $\bof_J \in \Co^{(\bv)}$ 
given by
\[
f_{J,j} :=
\begin{cases}
f_j, & \text{if $j < J$,}\\
v_j, & \text{otherwise,}
\end{cases}
\]
we have
\[
\lim_{J \to \infty} \|\KK(\cdot,\bof) - \KK(\cdot,\bof_J) \|_{H} = 0.
\]
\end{lemma}

\begin{proof}
Let $\bof \in \Co^{(\bv)}$. Assume at first that $\bof \in V$. 
Then we obtain
\[
1/2 \cdot \|\KK(\cdot,\bof) - \KK(\cdot,\bof_J) \|^2_{H}
= 1- \Re \prod_{j \geq J} \scp{v_j}{f_j}_{H_j}.
\]
This yields the convergence as claimed.

Assume now that $\KK(\cdot,\bof) \neq 0$.
Then we put $z_j := \|f_j\|_{H_j}$. For $\bof^\prime \in V$ according to
\eqref{g40} we already know that 
$\KK(\cdot,\bof_J^\prime)$ converges to $\KK(\cdot,\bof^\prime)$ in $H$.
Furthermore,
\[
\prod_{j \in \N} z_j \cdot
\|\KK(\cdot,\bof^\prime) - \KK(\cdot,\bof^\prime_J)\|_H =
\Bigl\|\KK(\cdot,\bof) - \prod_{j \geq J} z_j \cdot
\KK(\cdot,\bof_J)\Bigr\|_H
\]
and
\[
\Bigl\|\KK(\cdot,\bof_J) - \prod_{j \geq J} z_j \cdot
\KK(\cdot,\bof_J)\Bigr\|^2_H = \Bigl|1 - \prod_{j=J}^\infty z_j\Bigr|^2 
\cdot \prod_{j=1}^{J-1} z_j^2.
\]
This yields the convergence as claimed, since
$\sum_{j \in \N} |z_j-1| < \infty$ also in this case.

It remains to observe that $\KK(\cdot,\bof_J) = 0$ for $J$
sufficiently large, if $\KK(\cdot,\bof)=0$.
\end{proof}

Subsequently, we discuss further properties of incomplete tensor 
products, which may be derived with the help of Lemma~\ref{la6} 
in a straightforward way. 

First of all,
\[
H^{(\bv)} = 
\overline{\spann} \{\KK(\cdot,\bof) \colon \bof \in \Coo^{(\bv)}\},
\]
where
\[
\Coo^{(\bv)} := \{ \bof \in C_0 \colon 
\text{$\{j \in \N \colon f_j \neq v_j\}$ is finite}\} \subsetneq
\Co^{(\bv)}.
\]
Consequently, it suffices to consider the elements of $H^{(\bv)}$ on
the domain $\Coo^{(\bv)}$, since 
\begin{equation}\label{g47}
g(\bof) = \lim_{J \to \infty} g(\bof_J)
\end{equation}
for every $g \in H^{(\bv)}$ and every $\bof \in \Co^{(\bv)}$.

As a closed linear subspace of a reproducing kernel Hilbert space,
$H^{(\bv)}$ is a reproducing kernel Hilbert space, too.  
Moreover, the reproducing kernel $\Ko^{(\bv)} \colon C \times
C \to \K$ of $H^{(\bv)}$ satisfies 
\[
\Ko^{(\bv)}(\bog,\bof)
=
\begin{cases}
\KK(\bog,\bof), & \text{if $\bof,\bog \in \Co^{(\bv)}$},\\
0, & \text{otherwise}.
\end{cases}
\]
Observe that $\KK(\bog,\bof) = \prod_{j \in \N}
\scp{f_j}{g_j}_{H_j}$ is convergent, and not just quasi-con\-ver\-gent,
for $\bof,\bog \in \Co^{(\bv)}$.

Instead of $\Ko^{(\bv)}$ and $H^{(\bv)}$ one may actually 
consider the kernel
\[
\Koo^{(\bv)} := \KK|_{\Coo^{(\bv)} \times \Coo^{(\bv)}}
\]
and the space $H(\Koo^{(\bv)})$, since the restriction 
$g \mapsto g|_{\Coo^{(\bv)}}$ defines an isometric isomorphism between
$H^{(\bv)}$ and $H(\Koo^{(\bv)})$. The corresponding inverse
extension map is essentially given by \eqref{g47}.

\begin{rem}\label{onb}
Consider any choice of orthonormal bases $(e_{\nu,j})_{\nu \in N_j}$ 
in each of the spaces $H_j$, and assume that $0 \in N_j$ for
notational convenience. Let $\mathcal{N}$ denote the set of all 
sequences $\bn := (\nu_j)_{j \in \N}$ with $\nu_j \in N_j$ for every 
$j \in \N$ and with $\{j \in \N \colon \nu_j \neq 0\}$ being finite.
If 
\[
e_{0,j}=v_j
\]
for every $j \in \N$, then the elementary tensors 
\begin{equation}\label{g45}
e_\bn := \bigotimes_{j \in \N} e_{\nu_j,j}
\end{equation}
with $\bn \in \mathcal{N}$
form an orthonormal basis of $H^{(\bv)}$.
\end{rem}

\subsection{Tensor Products of Operators}\label{as6}

We now consider two incomplete tensor products
\[
H^{(\bv)} := \bigotimes_{j \in \N} H_j^{(v_j)}
\quad \text{and} \quad
G^{(\bw)} := \bigotimes_{j \in \N} G_j^{(w_j)}
\] 
together with a sequence of bounded linear operators 
$T_j \colon H_j \to G_j.$
We assume that $T_j v_j = w_j$ for every $j \in \N$ and that 
$\prod_{j \in \N} \|T_j\|$ converges. Observe that $T_j v_j = w_j$
implies $\| T_j \| \ge 1$. 
Letting 
\[
 T \bigotimes_{j \in \N} f_j  = \bigotimes_{j \in \N} T_j f_j 
\] 
for elementary tensors $\bigotimes_{j \in \N} f_j$ with 
$\bof \in \Coo^{(\bv)}$ and extending this linearly defines a 
linear operator from the span 
$H_0^{(\bv)}$ of these elementary tensors to 
$G^{(\bw)}$. 
This follows similarly as the corresponding statement for finite 
algebraic tensor products, see, e.g., \citet[Sec.~4.3.6]{Hac12}.
It is readily checked that this operator is bounded, in fact it 
has norm 
\[
 \| T \| = \prod_{j \in \N} \|T_j\|.
\]
Again, this follows similarly as for finite Hilbert space tensor 
products, see, e.g., \citet[Prop.~4.127]{Hac12}.
Since 
$H_0^{(\bv)}$ is dense in $H^{(\bv)}$,
$T$ uniquely extends to a linear 
operator $T \colon H^{(\bv)} \to G^{(\bw)}$ with the same norm.
This operator 
is called the \emph{tensor product} of 
the operators $T_j$ and denoted by $\bigotimes_{j \in \N} T_j$.
If, additionally, each operator $T_j$ is an isomorphism and 
$\prod_{j \in \N} \|T_j^{-1}\|$
converges as well, then $T \colon H^{(\bv)} \to G^{(\bw)}$ is an 
isomorphism with inverse $T^{-1}=\bigotimes_{j \in \N} T_j^{-1}$.

\subsection{Incomplete Tensor Products of Reproducing Kernel
Hilbert Spaces} \label{as4}

Here we consider the particular case 
\[
\phantom{\qquad\quad j \in \N,}
H_j := H(k_j),
\qquad\quad j \in \N,
\]
with reproducing kernels $k_j \colon D_j \times D_j \to \K$.
For every $j \in \N$ we assume that there exists a point $x \in
D_j$ with $k_j(\cdot,x) \neq 0$, which is equivalent to $H_j \neq \{0\}$.
We put $D := \times_{j \in \N} D_j$.

To $\bx \in D$ we associate 
\[
\tau (\bx) := (k_j(\cdot,x_j))_{j \in \N} \in \times_{j \in \N} H_j, 
\]
and we put
\[
\X_0 := \Bigl\{\bx \in D \colon 
\sum_{j \in \N} |k_j(x_j,x_j)-1| < \infty\Bigr\}
\]
as well as
\[
\Xo^{(\bv)} := \Bigl\{\bx \in \X_0 \colon 
\sum_{j \in \N} |v_j(x_j)-1| < \infty\Bigr\}
\]
for $\bv \in V$. 
Let $\bx \in D$.
Observe that $\bx \in \X_0$ if and only if $\tau(\bx) \in C_0$.
Moreover, $\bx \in \Xo^{(\bv)}$ if and only if $\tau(\bx) \in 
\Co^{(\bv)}$.

The following result is due to \citet[Lem.~4.9]{Rue20}.

\begin{lemma}\label{la7}
If $\Xo^{(\bv)} \neq \emptyset$, then
\[
H^{(\bv)} = 
\overline{\spann} 
\{ \KK(\cdot,\tau(\bx)) \colon \bx \in \Xo^{(\bv)}\}.
\]
\end{lemma}

\begin{proof}
Put $\HH := \spann 
\{ \KK(\cdot,\tau(\bx)) \colon \bx \in \Xo^{(\bv)}\}$.
It suffices to prove the following result.
For every $\bof \in \Coo^{(\bv)} \cap v$ there exists a sequence
$(f_J)_{J \in \N}$ in $\HH$ such that
\[
\lim_{J \to \infty} \|\KK(\cdot,\bof) - f_J\|_H = 0.
\]

For a fixed choice of $\bx \in \Xo^{(\bv)}$ with 
$k_j(\cdot,x_j) \neq 0$ for every $j \in \N$ we define $\bw \in
\Co^{(\bv)} \cap v$ by
\[
w_j := \|k_j(\cdot,x_j)\|^{-1}_{H_j} \cdot k_j(\cdot,x_j).
\]
Since $\bw \in \Co^{(\bw)}$, we get $\Co^{(\bv)} =
\Co^{(\bw)}$.
Furthermore, we define $\bof_J \in \Coo^{(\bw)} \cap v$ for $J \in \N$ by
\[
f_{J,j} :=
\begin{cases}
f_j, & \text{if $j \leq J$,}\\
w_j, & \text{otherwise.}
\end{cases}
\]
Lemma~\ref{la6} yields
\[
\lim_{J \to \infty} \|\KK(\cdot,\bof) - \KK(\cdot,\bof_J) \|_{H} = 0.
\]

Fix $J \in \N$ and consider the norm given by
$\|(h_1,\dots,h_J)\| :=
\max_{j=1,\dots,J} \|h_j\|_{H_j}$ on $H_1 \times \dots \times H_J$.
By $\zeta_J(h_1,\dots,h_J) = \KK(\cdot,\bog)$ with
\[
g_j :=
\begin{cases}
h_j, & \text{if $j \leq J$,}\\
w_j, & \text{otherwise,}
\end{cases}
\]
we obtain a continuous mapping $\zeta_J$ from 
$H_1 \times \dots \times H_J$.
into $H$.
For every $j \leq J$ we choose a sequence $(h_{j,n})_{n \in \N}$
in $\spann \{ k_j(\cdot,x) \colon x \in D_j\}$
with $\lim_{n \to \infty} \|f_j - h_{j,n}\|_{H_j} = 0$.
It follows that
\[
\lim_{n \to \infty}
\| \zeta_J(h_{1,n},\dots,h_{J,n}) - \KK(\cdot,\bof_J)\|_H = 0.
\]
Finally,
$\zeta_J(h_{1,n}, \dots, h_{J,n}) \in \HH$, which completes the
proof.
\end{proof}

Assume that $\X_0 \neq \emptyset$. 
In this case
\[
\phantom{\qquad\quad \bx,\by \in \X_0,}
K(\bx,\by) := 
\KK(\tau(\bx),\tau(\by)) 
= \prod_{j \in \N} k_j(x_j,y_j),
\qquad\quad \bx,\by \in \X_0,
\]
yields a reproducing kernel $K \colon \X_0 \times \X_0 \to \K$
of tensor product form. 
The complete tensor product 
$H$ of the reproducing kernel Hilbert spaces $H(k_j)$
is a feature space and $\bx \mapsto
\KK(\cdot,\tau(\bx))$ with $\bx \in \X_0$ is a feature map of $K$.
Consequently,
\[
H(K) = \{ g \circ \tau|_{\X_0} \colon g \in H\}
\]
and
\[
\|f\|_{H(K)} = \min \{ \|g\|_H \colon g \circ \tau|_{\X_0} = f\}.
\]

Assume that $\Xo^{(\bv)} \neq \emptyset$. Then analogous results 
hold true for the tensor product kernel 
\[
K^{(\bv)} := K|_{\Xo^{(\bv)} \times \Xo^{(\bv)}}
\]
and the incomplete tensor product $H^{(\bv)}$ of the reproducing
kernel Hilbert spaces $H(k_j)$.

In the sense of the following result, 
which is due to \citet[Thm.~4.10]{Rue20},
$H^{(\bv)}$ is
the reproducing kernel Hilbert space with the tensor product kernel 
$K^{(\bv)}$.

\begin{theo}\label{ta1}
Let $\bv \in V$.
If $\Xo^{(\bv)} \neq \emptyset$, then
\[
\Phi \colon H^{(\bv)} \to \K^{\Xo^{(\bv)}},
\]
given by
\[
\phantom{\qquad\quad \bx \in \Xo^{(\bv)},}
\Phi g (\bx) := g(\tau(\bx)),
\qquad\quad \bx \in \Xo^{(\bv)},
\]
induces an isometric isomorphism between $H^{(\bv)}$ and $H(K^{(\bv)})$.
In particular, for $\bof \in \Co^{(\bv)}$ and $\bx \in \Xo^{(\bv)}$
the product $\prod_{j \in \N} f_j(x_j)$ converges and
\[
\Bigg(\Phi \bigotimes_{j \in \N} f_j \Bigg) (\bx) = \prod_{j \in \N} f_j(x_j).
\]
\end{theo}

\begin{proof}
Due to Lemma~\ref{la7} the mapping $\Phi$ is injective,
and thus it induces an isometric isomorphism as claimed.
Furthermore, for $\bof \in \Co^{(\bv)}$ and $\bx \in \Xo^{(\bv)}$,
\[
\bigotimes_{j \in \N} f_j (\tau(\bx)) =
\scp{\KK(\cdot,\bof)}{\KK(\cdot,\tau(\bx)}_H = 
\prod_{j \in \N} f_j(x_j)
\]
with convergence as claimed, since $\tau(\bx) \in \Co^{(\bv)}$.
\end{proof}

\subsection{Incomplete Tensor Products of $L^2$-Spaces} \label{as5}

Here we consider the particular case
\[
\phantom{\qquad\quad j \in \N,}
H_j := L^2(\rho_j) := L^2(D_j,\rho_j),
\qquad\quad j \in \N,
\]
for a sequence $(\rho_j)_{j \in \N}$ of probability measures on 
$\sigma$-algebras on sets $D_j$. 
Let $\rho:=\times_{j \in \N} \rho_j$ denote the product measure on
the product $\sigma$-algebra on $D := \times_{j \in \N} D_j$.
Recall that we identify square-integrable
functions and the elements of the corresponding $L^2$-space. 
We now show 
that the incomplete tensor product $H^{(\be)}$ based on the 
constant function $1$ as the unit vector for every $j \in \N$
is, in a canonical way, isometrically isomorphic to $L^2(\rho) :=
L^2(D,\rho)$.

We introduce the following notation. 
For $\bof \in \Co^{(\be)}$ and $J \in \N$ we consider the 
function $\bx \mapsto \prod_{j=1}^J f_j(x_j)$, which belongs to 
$L^2(\rho)$. In the case of convergence in $L^2(\rho)$ as $J \to \infty$,
we use $\prod_{j=1}^\infty f_j$ to denote the limit; to indicate this 
convergence we say that $\prod_{j=1}^\infty f_j$ exists in $L^2(\rho)$.

See \citet[Corollary~6]{Gui69} for the following result.

\begin{theo}\label{ta2}
For every $\bof \in \Co^{(\be)}$ we have existence of
$\prod_{j=1}^\infty f_j$ in $L^2(\rho)$. Moreover,
the mapping $\Psi \colon L^2(\rho) \to \K^C$ given by
\[
\phantom{\qquad\quad \bog \in \Co^{(\be)},}
\Psi \gamma (\bog) := 
\int_D \gamma \cdot \prod_{j=1}^\infty \overline{g_j} \, d \rho,
\qquad\quad \bog \in \Co^{(\bu)},
\]
and $\Psi \gamma (\bog) := 0$ for
$\bog \in C \setminus \Co^{(\be)}$
defines an isometric isomorphism between $L^2(\rho)$ and $H^{(\be)}$.
In particular, 
\begin{equation}\label{g65}
\Psi^{-1} \bigotimes_{j \in \N} f_j = \prod_{j=1}^\infty f_j
\end{equation}
for $\bof \in \Co^{(\be)}$.
\end{theo}

\begin{proof}
The existence of $\prod_{j=1}^\infty f_j$ in $L^2(\rho)$ trivially
holds true for $\bof \in \Coo^{(\be)}$.
Let $\psi \colon \Coo^{(\be)} \to L^2(\rho)$ be given by
$\psi \bof := \prod_{j=1}^\infty f_j$.
Since $\scp{\psi \bof}{\psi\bog}_{L^2(\rho)} = \KK(\bog,\bof)$ 
for $\bof,\bog \in \Coo^{(\be)}$,
the mapping $\psi$ is a feature map and $L^2(\rho)$ 
is a feature space for the reproducing kernel $\Koo^{(\be)}$.
The space $H(\Koo^{(\be)})$ may therefore be
described via the linear mapping
$\Psi_1 \colon L^2(\rho) \to \K^{\Coo^{(\be)}}$,
given by
\[
\phantom{\quad\qquad \bog \in \Coo^{(\be)}.}
\Psi_1 \gamma (\bog) := \scp{\gamma}{\psi \bog}_{L^2(\rho)}
=
\int_D \gamma \cdot \prod_{j=1}^\infty \overline{g_j} \, d \rho,
\quad\qquad \bog \in \Coo^{(\be)}.
\]
Since $\spann \psi(\Coo^{(\be)})$ is a dense subset of $L^2(\rho)$,
the mapping $\Psi_1$ is injective. 
We conclude that $\Psi_1$ induces an isometric isomorphism 
between $L^2(\rho)$ and $H(\Koo^{(\be)})$, which will also
be denoted by $\Psi_1$. 

Recall that $h \mapsto h|_{\Coo^{(\be)}}$ defines an isometric
isomorphism between $H^{(\be)}$ and $H(\Koo^{(\be)})$. Its
inverse $\Psi_2 \colon H(\Koo^{(\be)}) \to H^{(\be)}$
is given by \eqref{g47}, i.e., by the extension
\[
\Psi_2 h (\bog) := 
\begin{cases}
\lim_{J \to \infty} h(\bog_J), & \text{if $\bog \in\Co^{(\be)}$},\\
0, & \text{if $\bog \in C \setminus \Co^{(\be)}$.}
\end{cases}
\]
Altogether, we obtain the isometric isomorphism
$\Psi := \Psi_2 \circ \Psi_1 \colon L^2(\rho) \to H^{(\be)}$.

Let $\gamma \in L^2(\rho)$.
By definition of $\Psi$ and $\bog_J$,
\[
\Psi \gamma (\bog) = 
\begin{cases}
\lim_{J \to \infty}
\int_D \gamma \cdot \prod_{j=1}^\infty \overline{g_{J,j}} \, d \rho,
& \text{if $\bog \in\Co^{(\be)}$,}\\
0, & \text{if $\bog \in C \setminus \Co^{(\be)}$.}
\end{cases}
\]
In the particular case $\gamma := \prod_{j=1}^\infty f_j$ with 
$\bof \in \Coo^{(\be)}$ we get
\[
\Psi \gamma (\bog) = \lim_{J \to \infty}
\int_D \prod_{j=1}^J f_j(x_j) \cdot \overline{g_j}(x_j) \, d \rho(\bx)
= \KK(\bog,\bof) = \bigotimes_{j \in \N} f_j (\bog)
\]
for $\bog \in \Co^{(\be)}$, i.e., \eqref{g65} is satisfied for
$\bof \in \Coo^{(\be)}$.
In the general case $\bof \in C^{(\be)}$ this implies
\[
\Psi^{-1} \bigotimes_{j \in \N} f_{J,j} (\bx)
= \prod_{j=1}^\infty f_{J,j} (\bx) = \prod_{j=1}^J f_j(x_j)
\]
for $\rho$-a.e.\ $\bx \in D$. Moreover,
$\bigotimes_{j \in \N} f_{J,j}$ converges to $\bigotimes_{j \in \N} f_j$ 
in $H^{(\be)}$ as $J \to \infty$, see Lemma~\ref{la6}. It follows that
\[
\lim_{J \to \infty} 
\biggl\|\prod_{j=1}^\infty f_{J,j} - 
\Psi^{-1} \bigotimes_{j \in \N} f_{j}\biggr\|_{L^2(\rho)} = 0,
\]
i.e., $\prod_{j=1}^\infty f_j$ exists in $L^2(\rho)$ and
\eqref{g65} is satisfied for $\bof \in \Co^{(\be)}$. Consequently,
with $\bog \in \Co^{(\be)}$ in place of $\bof$,
\[
\Psi \gamma (\bog) = 
\lim_{J \to \infty}
\int_D \gamma \cdot \prod_{j=1}^\infty \overline{g_{J,j}} \, d \rho
=
\int_D \gamma \cdot \prod_{j=1}^\infty \overline{g_j} \, d \rho
\]
for $\gamma \in L^2(\rho)$.
\end{proof}

\subsection*{Acknowledgment}

Part of this work was done during the Dagstuhl Seminar 19341
``Algorithms and Complexity for Continuous Problems'' and during the Research Meeting 20083
``Complexity of Infinite-Dimensional Problems'' 
at Schlo\ss\ Dagstuhl.
The authors would like to thank the staff of Schlo\ss\ Dagstuhl 
for their hospitality.
We also thank the anonymous referees for useful comments, which helped us to improve the presentation.
The research of Aicke Hinrichs was funded in part by the Austrian Science Fund (FWF) Projects F5513-N26 and P34808.
For the purpose of open access, the authors have applied a CC BY public copyright license to any Author Accepted Manuscript arising from this submission.

{\small

\bibliographystyle{abbrvnat}
\bibliography{references}

\begin{thebibliography}{37}
\providecommand{\natexlab}[1]{#1}
\providecommand{\url}[1]{\texttt{#1}}
\expandafter\ifx\csname urlstyle\endcsname\relax
  \providecommand{\doi}[1]{doi: #1}\else
  \providecommand{\doi}{doi: \begingroup \urlstyle{rm}\Url}\fi

\bibitem[Baldeaux and Gnewuch(2014)]{BG12}
J.~Baldeaux and M.~Gnewuch.
\newblock Optimal randomized multilevel algorithms for infinite-dimensional
  integration on function spaces with {ANOVA}-type decomposition.
\newblock \emph{SIAM J.~Numer. Anal.}, 52:\penalty0 1128--1155, 2014.

\bibitem[Bogachev(1998)]{MR1642391}
V.~I. Bogachev.
\newblock \emph{Gaussian Measures}, volume~62 of \emph{Mathematical Surveys and
  Monographs}.
\newblock American Mathematical Society, Providence, RI, 1998.

\bibitem[Chandler-Wilde et~al.(2015)Chandler-Wilde, Hewett, and
  Moiola]{ChaEtAl2015}
S.~N. Chandler-Wilde, D.~P. Hewett, and A.~Moiola.
\newblock Interpolation of {H}ilbert and {S}obolev spaces: quantitative
  estimates and counterexamples.
\newblock \emph{Mathematika}, 61:\penalty0 414--443, 2015.

\bibitem[Dette and Zhigljavsky(2021)]{DZ21}
H.~Dette and A.~Zhigljavsky.
\newblock Reproducing kernel {H}ilbert spaces, polynomials, and the classical
  moment problem.
\newblock \emph{SIAM/ASA J.\ Uncertain.\ Quantif.}, 9:\penalty0 1589--1614,
  2021.

\bibitem[Dick and Gnewuch(2014)]{DG13}
J.~Dick and M.~Gnewuch.
\newblock Optimal randomized changing dimension algorithms for
  infinite-dimensional integration on function spaces with {ANOVA}-type
  decomposition.
\newblock \emph{J. Approx. Theory}, 184:\penalty0 111--145, 2014.

\bibitem[Dick et~al.(2018)Dick, Irrgeher, Leobacher, and
  Pillichshammer]{DILP18}
J.~Dick, C.~Irrgeher, G.~Leobacher, and F.~Pillichshammer.
\newblock On the optimal order of integration in {H}ermite spaces with finite
  smoothness.
\newblock \emph{SIAM J.~Numer. Anal.}, 56:\penalty0 684--707, 2018.

\bibitem[Erd\'{e}lyi et~al.(1953)Erd\'{e}lyi, Magnus, Oberhettinger, and
  Tricomi]{HTF1953}
A.~Erd\'{e}lyi, W.~Magnus, F.~Oberhettinger, and F.~G. Tricomi.
\newblock \emph{Higher transcendental functions. {V}ol. {II}}.
\newblock McGraw-Hill Book Company, Inc., New York-Toronto-London, 1953.
\newblock Based, in part, on notes left by Harry Bateman.

\bibitem[Fasshauer et~al.(2012)Fasshauer, Hickernell, and
  Wo\'zniakowski]{FHW2012}
G.~Fasshauer, F.~J. Hickernell, and H.~Wo\'zniakowski.
\newblock On dimension-independent rates of convergence for function
  approximation with {G}aussian kernels.
\newblock \emph{SIAM J. Numer. Anal.}, 50:\penalty0 247--271, 2012.

\bibitem[Gilbert et~al.(2018)Gilbert, Kuo, Nuyens, and
  Wasilkowski]{GilEtAl2018}
A.~D. Gilbert, F.~Y. Kuo, D.~Nuyens, and G.~W. Wasilkowski.
\newblock Efficient implementations of the multivariate decomposition method
  for approximating infinite-variate integrals.
\newblock \emph{SIAM J.~Scient.~Comput.}, 40:\penalty0 A3240--A3266, 2018.

\bibitem[Gnewuch et~al.(2017)Gnewuch, Hefter, Hinrichs, and Ritter]{GneEtAl16}
M.~Gnewuch, M.~Hefter, A.~Hinrichs, and K.~Ritter.
\newblock Embeddings of weighted {H}ilbert spaces and applications to
  multivariate and infinite-dimensional integration.
\newblock \emph{Journal of Approximation Theory}, 222:\penalty0 8--39, 2017.

\bibitem[Gnewuch et~al.(2019)Gnewuch, Hefter, Hinrichs, Ritter, and
  Wasilkowski]{GHHRW2020}
M.~Gnewuch, M.~Hefter, A.~Hinrichs, K.~Ritter, and G.~W. Wasilkowski.
\newblock Embeddings for infinite-dimensional integration and
  ${L}_2$-approximation with increasing smoothness.
\newblock \emph{J.\ Complexity}, 54:\penalty0 101406, 2019.

\bibitem[Guichardet(1969)]{Gui69}
A.~Guichardet.
\newblock \emph{Tensor Products of $C^*$-Algebras, Part II. Infinite Tensor
  Products}, volume~13 of \emph{Lecture Notes Series}.
\newblock Aarhus Universitet, 1969.

\bibitem[Hackbusch(2012)]{Hac12}
W.~Hackbusch.
\newblock \emph{Tensor spaces and numerical tensor calculus}, volume~42 of
  \emph{Springer Series in Computational Mathematics}.
\newblock Springer, Heidelberg, 2012.

\bibitem[Hickernell and Wang(2001)]{HicWan2001}
F.~J. Hickernell and X.~Wang.
\newblock The error bounds and tractability of quasi-{M}onte {C}arlo algorithms
  in infinite dimensions.
\newblock \emph{Math.\ Comp.}, 71:\penalty0 1614--1661, 2001.

\bibitem[Hickernell et~al.(2010)Hickernell, M\"{u}ller-Gronbach, Niu, and
  Ritter]{HicEtAl2010}
F.~J. Hickernell, T.~M\"{u}ller-Gronbach, B.~Niu, and K.~Ritter.
\newblock Multi-level {M}onte {C}arlo algorithms for infinite-dimensional
  integration on {$\mathbb R^{\mathbb N}$}.
\newblock \emph{J. Complexity}, 26\penalty0 (3):\penalty0 229--254, 2010.

\bibitem[Indritz(1961)]{MR132852}
J.~Indritz.
\newblock An inequality for {H}ermite polynomials.
\newblock \emph{Proc. Amer. Math. Soc.}, 12:\penalty0 981--983, 1961.

\bibitem[Irrgeher and Leobacher(2015)]{IL15}
C.~Irrgeher and G.~Leobacher.
\newblock High-dimensional integration on the $\mathbb{R}^d$, weighted
  {H}ermite spaces, and orthogonal transforms.
\newblock \emph{J. Complexity}, 31:\penalty0 174--205, 2015.

\bibitem[Irrgeher et~al.(2015)Irrgeher, Kritzer, Leobacher, and
  Pillichshammer]{IKLP15}
C.~Irrgeher, P.~Kritzer, G.~Leobacher, and F.~Pillichshammer.
\newblock Integration in {H}ermite spaces of analytic functions.
\newblock \emph{J. Complexity}, 31:\penalty0 380--404, 2015.

\bibitem[Irrgeher et~al.(2016{\natexlab{a}})Irrgeher, Kritzer, Pillichshammer,
  and Wo\'zniakowski]{IKPW16a}
C.~Irrgeher, P.~Kritzer, F.~Pillichshammer, and H.~Wo\'zniakowski.
\newblock Tractability of multivariate approximimation defined over {H}ilbert
  spaces with exponential weights.
\newblock \emph{J. Approx. Theory}, 207:\penalty0 301--338, 2016{\natexlab{a}}.

\bibitem[Irrgeher et~al.(2016{\natexlab{b}})Irrgeher, Kritzer, Pillichshammer,
  and Wo\'zniakowski]{IKPW16b}
C.~Irrgeher, P.~Kritzer, F.~Pillichshammer, and H.~Wo\'zniakowski.
\newblock Approximation in {H}ermite spaces of smooth functions.
\newblock \emph{J. Approx. Theory}, 207:\penalty0 98--126, 2016{\natexlab{b}}.

\bibitem[Karvonen(2021)]{Karvonen21}
T.~Karvonen.
\newblock On non-inclusion of certain functions in reproducing kernel {H}ilbert
  spaces.
\newblock arXiv:2102.10628, 2021.

\bibitem[Karvonen et~al.(2021)Karvonen, Oates, and Girolami]{KOG21}
T.~Karvonen, C.~J. Oates, and M.~Girolami.
\newblock Integration in reproducing kernel {H}ilbert spaces of {G}aussian
  kernels.
\newblock arXiv:2004.12654v2, 2021.

\bibitem[Kuo et~al.(2010)Kuo, Sloan, Wasilkowski, and
  Wo{\'z}niakowski]{KuoEtAl10}
F.~Y. Kuo, I.~H. Sloan, G.~W. Wasilkowski, and H.~Wo{\'z}niakowski.
\newblock Liberating the dimension.
\newblock \emph{J. Complexity}, 26:\penalty0 422--454, 2010.

\bibitem[Kuo et~al.(2017)Kuo, Sloan, and Wo\'zniakowski]{KSW2017}
F.~Y. Kuo, I.~H. Sloan, and H.~Wo\'zniakowski.
\newblock Multivariate integration for analytic functions with {G}aussian
  kernels.
\newblock \emph{Math. Comp.}, 86:\penalty0 829--853, 2017.

\bibitem[Papageorgiou and Wo{\'z}niakowski(2010)]{PapWoz10}
A.~Papageorgiou and H.~Wo{\'z}niakowski.
\newblock Tractability through increasing smoothness.
\newblock \emph{J. Complexity}, 26\penalty0 (5):\penalty0 409--421, 2010.

\bibitem[Plaskota and Wasilkowski(2011)]{PW11}
L.~Plaskota and G.~W. Wasilkowski.
\newblock Tractability of infinite-dimensional integration in the worst case
  and randomized settings.
\newblock \emph{J. Complexity}, 27:\penalty0 505--518, 2011.

\bibitem[Rasmussen and Williams(2006)]{RW2006}
C.~E. Rasmussen and C.~K.~I. Williams.
\newblock \emph{{G}aussian Processes for Machine Learning}.
\newblock MIT Press, Cambridge, 2006.

\bibitem[Rodino(1993)]{R93}
L.~Rodino.
\newblock \emph{Linear partial differential operators in {G}evrey spaces}.
\newblock World Scientific Publishing Co., Inc., River Edge, NJ, 1993.

\bibitem[R\"u{\ss}mann(2020)]{Rue20}
R.~R\"u{\ss}mann.
\newblock Tensor products of {H}ilbert spaces.
\newblock Master's thesis, Department of Mathematics, TU Kaiserslautern, 2020.

\bibitem[Sloan and Wo\'zniakowski(1998)]{SW98}
I.~H. Sloan and H.~Wo\'zniakowski.
\newblock When are quasi-{M}onte {C}arlo algorithms efficient for high
  dimensional integrals?
\newblock \emph{J. Complexity}, 14:\penalty0 1--33, 1998.

\bibitem[Sloan and Wo\'zniakowski(2018)]{SW2018}
I.~H. Sloan and H.~Wo\'zniakowski.
\newblock Multivariate approximation for analytic functions with {G}aussian
  kernels.
\newblock \emph{J. Complexity}, 45:\penalty0 1--21, 2018.

\bibitem[Steinwart and Christmann(2008)]{SC08}
I.~Steinwart and A.~Christmann.
\newblock \emph{Support vector machines}.
\newblock Information Science and Statistics. Springer-Verlag, New York, 2008.

\bibitem[Steinwart et~al.(2006)Steinwart, Hush, and Scovel]{SHS2006}
I.~Steinwart, D.~Hush, and C.~Scovel.
\newblock An explicit description of the reproducing kernel {H}ilbert space of
  {G}aussian {RBF} kernels.
\newblock \emph{IEEE Trans. Inform. Theory}, 52:\penalty0 4635--4663, 2006.

\bibitem[Szeg{\" o}(1975)]{S75}
G.~Szeg{\" o}.
\newblock \emph{Orthogonal Polynomials}.
\newblock American Mathematical Society, Providence, R.I., fourth edition,
  1975.
\newblock American Mathematical Society, Colloquium Publications, Vol. XXIII.

\bibitem[Triebel(1978)]{MR503903}
H.~Triebel.
\newblock \emph{Interpolation theory, function spaces, differential operators},
  volume~18 of \emph{North-Holland Mathematical Library}.
\newblock North-Holland Publishing Co., Amsterdam-New York, 1978.

\bibitem[von Neumann(1939)]{Neu39}
J.~von Neumann.
\newblock On infinite direct products.
\newblock \emph{Compositio Math.}, 6:\penalty0 1--77, 1939.

\bibitem[Wasilkowski(2012)]{Was12}
G.~W. Wasilkowski.
\newblock Liberating the dimension for {$L_2$}-approximation.
\newblock \emph{J. Complexity}, 28\penalty0 (3):\penalty0 304--319, 2012.

\end{thebibliography}

}

\end{document}